\RequirePackage{fix-cm}
\documentclass[smallextended]{svjour3}       
\smartqed  
\usepackage{graphicx, epsf}
\usepackage{latexsym,amssymb} 

\begin{document}

\title{Optimal Fillings
}
\subtitle{A new spatial subdivision problem related to packing and covering}


\author{Carolyn L. Phillips         \and
        Joshua A. Anderson \and Elizabeth R. Chen \and Sharon C. Glotzer 
}


\institute{C. L. Phillips  \at
               Applied Physics, University of Michigan
                \and
           J. A. Anderson  
           \at
               Chemical Engineering, University of Michigan
           \and
            E. R. Chen 
                       \at
               Chemical Engineering, University of Michigan
            \and
	     S. C. Glotzer
	     \at
              Applied Physics, Department of Materials Science and Engineering, Chemical Engineering, University of Michigan\\
               \\
              \email{sglotzer@umich.edu}           
}

\date{13 August 2012}

\maketitle

\begin{abstract}
We present {\it filling} as a new type of spatial subdivision problem that is related to covering and packing. Filling addresses the optimal placement of overlapping objects lying entirely inside an arbitrary shape so as to cover the most interior volume.  In $n$-dimensional space, if the objects are polydisperse $n$-balls, we show that solutions correspond to sets of maximal $n$-balls and the solution space can reduced to the medial axis of a shape.  We examine the structure of the solution space in two dimensions. For the filling of polygons, we provide detailed descriptions of a heuristic and a genetic algorithm for finding solutions of maximal discs.  We also consider the properties of ideal distributions of $N$ discs in polygons as $N\rightarrow \infty$. 

\end{abstract}

\section{Introduction}
\label{intro}

First introduced in reference \cite{phillips}, we define \emph{filling} as the problem of packing overlapping objects inside of a defined shape such as to optimally cover the interior volume without extending beyond the boundary of the shape.  We are primarily interested in the optimal filling of an $n$-dimensional shape with a well-defined $n-1$ surface with $n$-dimensional polydisperse balls.

The filling problem can be expressed by the following two questions:

\begin{problem}
Given a compact region $G$ (having non-empty interior and no holes in the interior) and a fixed positive integer $N$, how can $N$ balls of varying radii be placed completely interior to $G$ so as to  maximize the total volume covered?  Overlaps of the $N$ balls are permitted.
\end{problem}

\begin{problem}
In general, for each fixed shape $G$, what is the best strategy for maximizing the fraction of volume covered by the minimum number of balls of varying radii in G?
\end{problem}


In the deceptively simple problem of determining the optimal set of balls to fill an arbitrary shape we find a surprisingly rich problem with many open questions.   In this paper we address the above two questions. In Section \ref{sec:genprop} we define the filling problem and the basic terminology necessary to discuss the problem.  We describe the mathematical structure of the filling solution space and show that it can be reduced from dimension $n+1$ to dimension $n-1$.  We characterize the forms of degeneracy possible in filling solutions and demonstrate how each individual ball contributes to the filling of a shape.  In Section \ref{sec:2D}, we first restrict the filling problem to planar shapes.  We then identify features of the solution space that affect the properties of an optimal solution and the methods for finding solutions.  We introduce the concept of neighbors and show how a fixed point in the solution space divides the solution space into independent spaces.  We show how special points in the solution space are found in many solutions.   We then restrict the problem further to polygons, where the structure of the solution space can be reduced to a small number of cases.  We numerically explore the solution space of a simple construction of three discs, and show the solution space is complex with many local maxima and topologically diverse configurations.  In Section \ref{sec:algs} we detail two algorithms for generating filling solutions for polygons, a genetic algorithm and a heuristic algorithm.  The genetic algorithm utilizes a minimal set of assumptions about the solution space.  The heuristic algorithm exploits the known structure of the solution space and also relies on several conjectures.  We discuss the relative efficiency of the heuristic algorithm in searching the solution space and the good correspondence between the two algorithms. In Section \ref{sec:contlim} we find the distribution of discs in a polygon at the continuum limit, or as $N\rightarrow\infty$.  The derived analytical expressions may be used to approximate solutions for finite but large $N$. We derive an expression for the fractional allocation of discs over the medial axis branches an arbitrary polygon and an exact expression for a triangle as $N\rightarrow \infty$.   In Section \ref{sec:conclusion}, we provide concluding remarks.  In Section \ref{sec:glossary}, we include a glossary of terms defined in this paper.

\section{General Properties of the Medial Axis of $G$ and Filling Solutions} \label{sec:genprop}

\subsection{Definitions and Theorems}

Let $G$ be a compact (closed and bounded), simply-connected $n$-dimensional region with a non-empty interior.  Let $S$ be the boundary of G: $S$ = $\delta$ $G$.  As in reference \cite{blum1978}, we restrict $S$ to have a tangent and curvature defined everywhere but at a finite number of points.  At these points, sided curvature, i.e. a limit performed only on one side of the point, exists from any direction along the boundary.  For simplicity, we do not consider $G$ with holes, nor $G$ with a boundary that abuts itself.

\begin{definition}  Let $R_N$ be a set containing $N$ balls $D_i$ that are completely contained in $G$.  Each $D_i$ has a radius $r_i$ and center $\bf{x}_i$.  $R_N$ is a \emph{filling} solution of $G$.  Let $\phi(R_N,G)$ be the fraction of $G$ that is covered by $R_N$. $\phi$ is the measure of the filling  $G$ by the set $R_N$.
\end{definition}

The measure $\phi$ is equal to the volume of the union of balls of $R_N$ divided by the volume of $G$, thus $\phi \leq 1$ by definition.  $\phi$ is equal to unity for $N < \infty$ only if $G$ is equivalent to a finite number of overlapping balls.   The space of all $R_N$ is of dimension $n+1$.

\begin{definition} If, $\forall D_i \in R_N$, $\phi(R_N-\{D_i\},G) < \phi(R_N,G)$, then the set is \emph{all-filling}.   In other words, each ball in $R_N$ uniquely fills a non-zero volume of $G$.
\end{definition}
 
\begin{definition}
 A set $R_N$ is an \emph{optimal filling} solution of $G$ if there is no other set $R'_N$ that satisfies $\phi(R'_N, G) > \phi(R_N,G)$.
\end{definition}

The function $\phi$ can be defined over the space of all $R_N$ for a shape $G$ and fixed positive $N$.  Our objective is to find the set $R_N$ with the maximum value of $\phi$.   We will now prove that the solution space can be restricted to sets of $R_N$ containing only \emph{maximal} $n$-\emph{balls}, a space defined by the  \emph{medial axis} of $G$ and its associated \emph{radius function}.

 The medial axis of an object, originally defined by Blum in reference \cite{blum1967},  and also known as the topological or medial skeleton, is the set of all points having more than one closest point on the object's boundary.  We use the notation $M(G)$ for the medial axis of $G$.  The medial axis is a reduction of an $n$-dimensional shape into an $n-1$-dimensional space, the locus of centers of the maximal $n$-balls.   A maximal $n$-ball is defined as a ball that is tangent to the boundary at two or more points.  It is also a ball contained completely in $G$ that is not a proper subset of any other ball also contained in $G$.   A shape is the logical union of all its maximal $n$-balls.    The radius function associated with $M(G)$ is a continuous, non-negative function defined at each point of $M(G)$ as the radius of the maximal $n$-ball centered at that point.  The medial axis and the radius function together are a complete shape descriptor \cite{blum1978} and can be used to reconstruct the shape.

\begin{theorem}
For $G$, there exists optimal filling solutions $R_N$ that contain only maximal balls. \label{thm:maximal}
\end{theorem}

\begin{proof}
From any filling solution set that has a ball that is not on the medial axis, we can construct a solution set that contains only balls on the medial axis.
Assume we have a solution set $R_N$ that contains a ball $D$ that is not tangent to $S$ at any point.  That ball is completely contained inside a concentric ball that is tangent to at least one point of $S$.   And that ball is completely contained inside a larger cotangent ball that is also tangent to a second point of $S$.   This last ball $D'$ has its center on some part of the medial axis by construction and is thus a maximal ball.  Let $R'_N$ be the set of balls where $D$ is replaced by $D'$.  It must be that $\phi(R'_N,G)\geq \phi(R_N,G)$.  So if $R_N$ is an optimal filling solution, then so is $R'_N$.
\end{proof}

While Theorem \ref{thm:maximal} implies that filling solutions can be restricted to sets of maximal balls, it does not follow that optimal solutions must be composed of maximal balls for all shapes.  Shapes that have boundaries with concave points of infinite curvature can have optimal filling solutions with non-maximal balls.   Figure \ref{oddcases}(a) shows such a shape.  The outer boundary of the shape in Figure \ref{oddcases}(a) is defined by four circular arcs.  The dashed (green online) line is the medial axis of this shape.  If the four discs at the extreme points of the medial axis have been placed, then there is no need for the final disc placed inside the shape to be a maximal disc.  

Even when restricted to optimal solutions of maximal balls, the entire medial axis need not be occupied as $N\rightarrow \infty$.   Figure \ref{oddcases}(b) is an example of a concave shape with two concave points.  The portion of the medial axis between the two circle centers (red online) need not be occupied by disc centers to fill the shape as $N\rightarrow \infty$.

\begin{theorem} If $S$ contains no concave points of infinite curvature, then optimal fillings $R_N$ composed of maximal balls are also all-filling.  Only filling  solutions of maximal balls can be optimal.  The entire medial axis is occupied for optimal filling solutions as $N \rightarrow \infty$. \label{thm:onlymaximal}
\end{theorem}

\begin{proof}
Assume there is an all-filling filling solution $R_N$ for a shape $G$ that contains a ball $D$ that is not a maximal ball.  The operations from Theorem \ref{thm:maximal} are used to construct a maximal ball $D'$ on the medial axis from the ball $D$.    The disc $D'$, by construction, is tangent to the boundary of $S$ in at least one location that $D$ was not.  Assuming that $S$ has no concave points of infinite curvature (e.g a reflex vertex for a polytope), then the point of tangency is smooth and there is a small region around the point of tangency that $D'$ covers that $D$ did not (Figure \ref{fig:tangency}).  That region can only already have been covered by a ball in $R_N$ if a ball contained in $R_N$ of equal or larger radius, was tangent to $S$ the same point.  But since disc $D'$, and therefore $D$, would be completely contained in that ball, then $R_N$ would not be all-filling.  So $\phi(R'_N,S)$ is strictly greater than $\phi(R_N,S)$.  If each point on the boundary $S$ is tangent to a maximal ball at a smooth point, then there is one maximal ball tangent to S at that point, so each point of S maps to a one and only one point on the medial axis.  Also, a unique infinitesimal volume is covered by that maximal ball, so the entire medial axis must be occupied for the optimal fillings as $N \rightarrow \infty$.
\end{proof}

\begin{figure*}[htb]
\begin{center}
\includegraphics[width=0.75\textwidth]{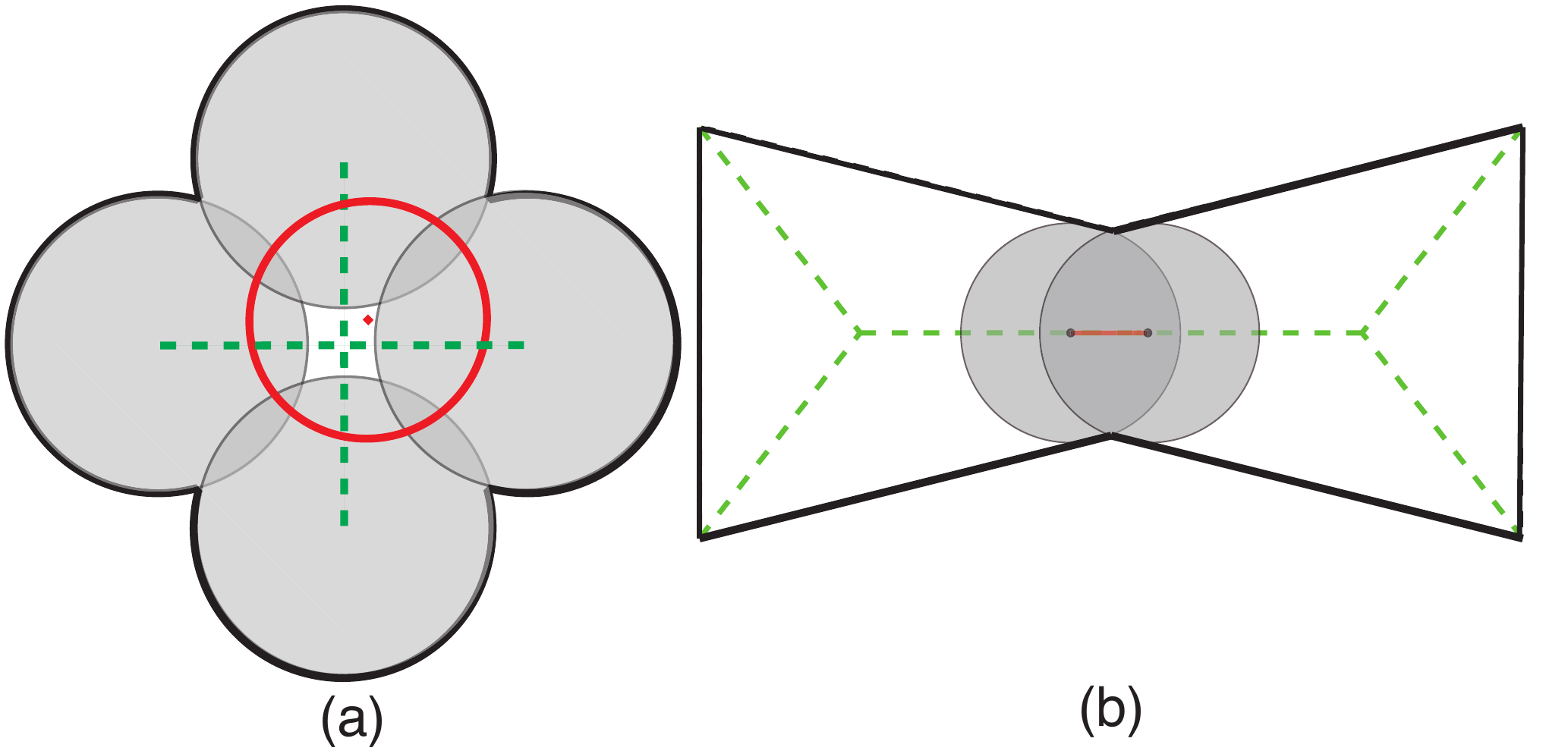}
\caption{(a)The construction has optimal solutions without maximal balls.  The center of the red disc need not be on the medial axis (dashed green) for the shape to be completely covered.  (b) The construction need not have all of its medial axis (dashed green) filled.  A disc added to the red portion fills no additional area.}
\label{oddcases}
\end{center}\end{figure*}

\begin{figure*}[htb]
\begin{center}
\includegraphics[width=0.75\textwidth]{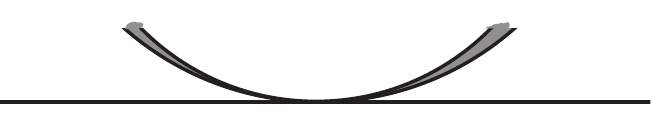}
\caption{ If the point of tangency is smooth in any direction, the largest ball tangent to the point covers more volume than any smaller ball tangent to the point. \label{fig:tangency}}
\end{center}\end{figure*}

Theorem \ref{thm:maximal} shows that to construct filling solutions, the search space can be restricted to the space of maximal balls.  Theorem \ref{thm:onlymaximal} shows that for $G$ with $S$ without concave points of infinite curvature, optimal filling solutions consist only of maximal balls.   Searching for optimal fillings has been reduced from finding points in an $n+1$ dimensional space (disc center position and radius) to finding points on an $n-1$ dimensional surface, or a problem of dimension $n-1$.

In practice, this surface is better described as a set of bounded connected surfaces.  For example, a planar shape has a medial axis that forms a planar graph, a connected set of curves that meet at points.  In three dimensions, a medial axis is composed of sheets, seams, and junctions\cite{Sheehy,Culver,Culver2004}.    

In the following sections we shall implicitly assume all $R_N$  solutions being discussed are all-filling.

\begin{figure*}[htb]
\begin{center}
\includegraphics[width=0.75\textwidth]{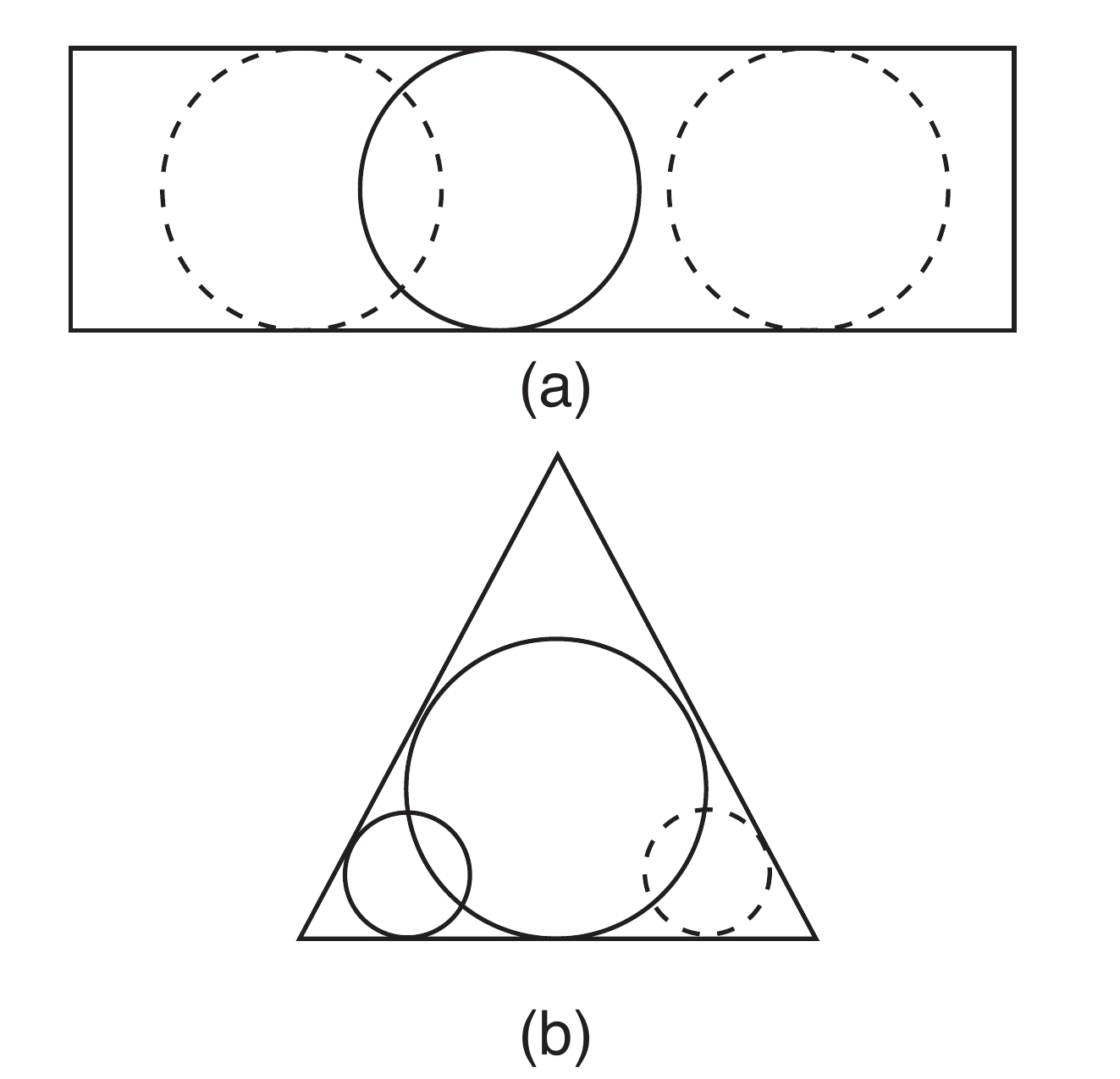}
\caption{Two examples of degenerate solutions.  For (a) a single maximal ball can be placed in infinitely many locations in a rectangle. For (b), the symmetrical triangle, the asymmetrical solution can be reflected to create a degenerate solution.} \label{fig:degenerate}
\end{center}\end{figure*}

For an arbitrary shape G, there is no guarantee that optimal filling solutions must be unique.   For example, Figure \ref{fig:degenerate} shows how a long rectangle can have an infinite number of $N=1$ solutions.  All discs added to the middle of the rectangle have the same diameter.  When sufficiently many discs have been added to the rectangle so that the discs must overlap, this form of degeneracy disappears.  Degenerate solutions also occur in symmetrical shapes with asymmetrical optimal filling  solutions, such as the $N=2$ solution shown for the triangle in Figure \ref{fig:degenerate} .  This form of degeneracy does not disappear as $N \rightarrow \infty$.  It is also possible for a shape with no symmetries to have two distinct optimal filling  solutions with the same filling. The likelihood of two distinctly different sets of discs both to be optimal and have the same $\phi$ is unlikely, and thus this form of degeneracy is also likely to be extremely rare.

\subsection{Filling contribution of a single ball}

The contribution of a single ball to the measure $\phi$ is equal to the volume of the ball offset by its fractional share of the volume of any overlap with other balls. 
More explicitly, the contribution of a single ball is the volume it uniquely covers plus $1/i$ of the volume it shares with exactly $i$ other balls. Let $V'_i(D_k)$ be the domain that a ball $D_k$ shares with exactly $i$ other balls, including itself.  Let $\mathcal{M}_V(V)$ be the measure of the volume of a domain.  Note that $V'_1(D_k)$ is the domain uniquely covered by the ball $D_k$.  The contribution $\mathcal{C}(D_k)$ is,

\begin{equation}
\mathcal{C}(D_k) = \sum^\infty_{i = 1} \frac{1}{i} \mathcal{M}_V\left(V'_i(D_k)\right)
\end{equation}
and,
\begin{equation}
\phi(R_N) = \frac{1}{\mathcal{M}_V(G)} \sum^{N}_{k=1} \mathcal{C}(D_k) \label{eqn:totalphi}.
\end{equation}

In general, if the overlap between balls $P$ and $Q$ is completely contained inside of the overlap between balls $Q$ and $R$, then locally  adding, removing, or displacing ball $P$ cannot uncover any of that overlap volume.  So if ball $P$ is added, removed, or locally displaced, the contribution of other balls may change but the only term that changes in the summed contributions of all the balls is $\mathcal{M}_V\left(V'_1(P)\right)$.

\section{Planar Shapes and Polygons} \label{sec:2D}
We now restrict the problem to planar shapes, $G$, whose medial axes are the locus of the centers of maximal discs.  Various algorithms exist to compute the medial axis of simple polygons (polygons that are not self-intersecting) and planar regions bounded by line segments, circular arcs, and general nonuniform rational B-spline~\cite{Vilaplana,preparata,gandp}.

We define the terminology and review the properties for an $M(G)$ of a planar shape introduced by Blum and Nagel~\cite{blum1978}.
$M(G)$ consists of connected subsets of points that form a 1-D planar graph.  Most points of $M(G)$ are \emph{normal points}, whose maximal disc is in contact with the boundary at two separate but contiguous sets of points.  $M(G)$ also contains a finite number of \emph{branch points}, each of which has a maximal disc in contact with the boundary at three or more separate but contiguous sets of points, and a finite number of \emph{end points}, whose maximal disc is in contact with the boundary at only one contiguous set of points.  For all but a finite number of discs, the contiguous set of points is a single point of contact.  The contact point consists of more than just a single point if the radius of curvature  the disc and boundary are the same.  As long as $G$ has no holes, then the graph $M(G)$ forms a tree with no loops.   $M(G)$ can be divided into sets of contiguous normal points bounded by branch or end points, such that the division is unique, disjoint, and complete.  Sets of contiguous normal points shall be referred to as a \emph{branch}.  The boundary of $S$ can be divided into parts associated with each branch by the intersection of $S$ with the set of maximal discs defined over a branch.  This division of $S$ is also unique, disjoint, and complete\cite{blum1978}.  The shape and radius function of any branch of $M(G)$ is determined by the \emph{parents} of the branch, that is, the two contiguous sections of $S$ associated with the given branch. Given a sequence of maximal discs on the branch, their respective points of contact with $S$ and their centers on the branch are traversed in the same order (assuming the branch is traversed in the correct direction relative to $S$).   

\begin{definition}
Given a filling solution $R_N$, and $D \in R_N$ for a $M(G)$ planar graph, we define the \emph{neighbors} of a maximal disc $D$, as the maximal discs whose centers are the closest along paths in $M(G)$ originating at the center of $D$.  
\end{definition}
In other words, if there is a path in $M(G)$ that connects the center of disc $D$ to the center of disc $D'$ without traversing another disc center, then disc $D$ and disc $D'$ are neighbors (Figure \ref{fig:neighbors}).  For $G$ with no holes, where $M(G)$ has no loops, there is only one path connecting any two disc centers. When a branch of the medial axis is populated with many maximal discs, most of the discs have exactly two neighbors, the disc to the left and right of them in sequence.  If a branch point has connectivity $n$, then, in a densely populated $M(G)$ the disc closest to that branch point has $n$ neighbors.  A disc can theoretically have as many neighbors as $M(G)$ has end points with connectivity 1.  

\begin{theorem} Any overlap of disc $D$ with any disc that is not a neighbor of $D$ must be contained inside the overlap of $D$ with one of its neighbors.  \label{thm:neighbors}
\end{theorem}

Thus to measure $\mathcal{M}_V\left(A'_1(D)\right)$ for a disc $D$, only the position of the neighbors of $D$ need be accounted for.  To show the latter is true, we draw upon the properties of a planar medial axis.

\begin{figure*}[htb]
\begin{center}
\includegraphics[width=0.75\textwidth]{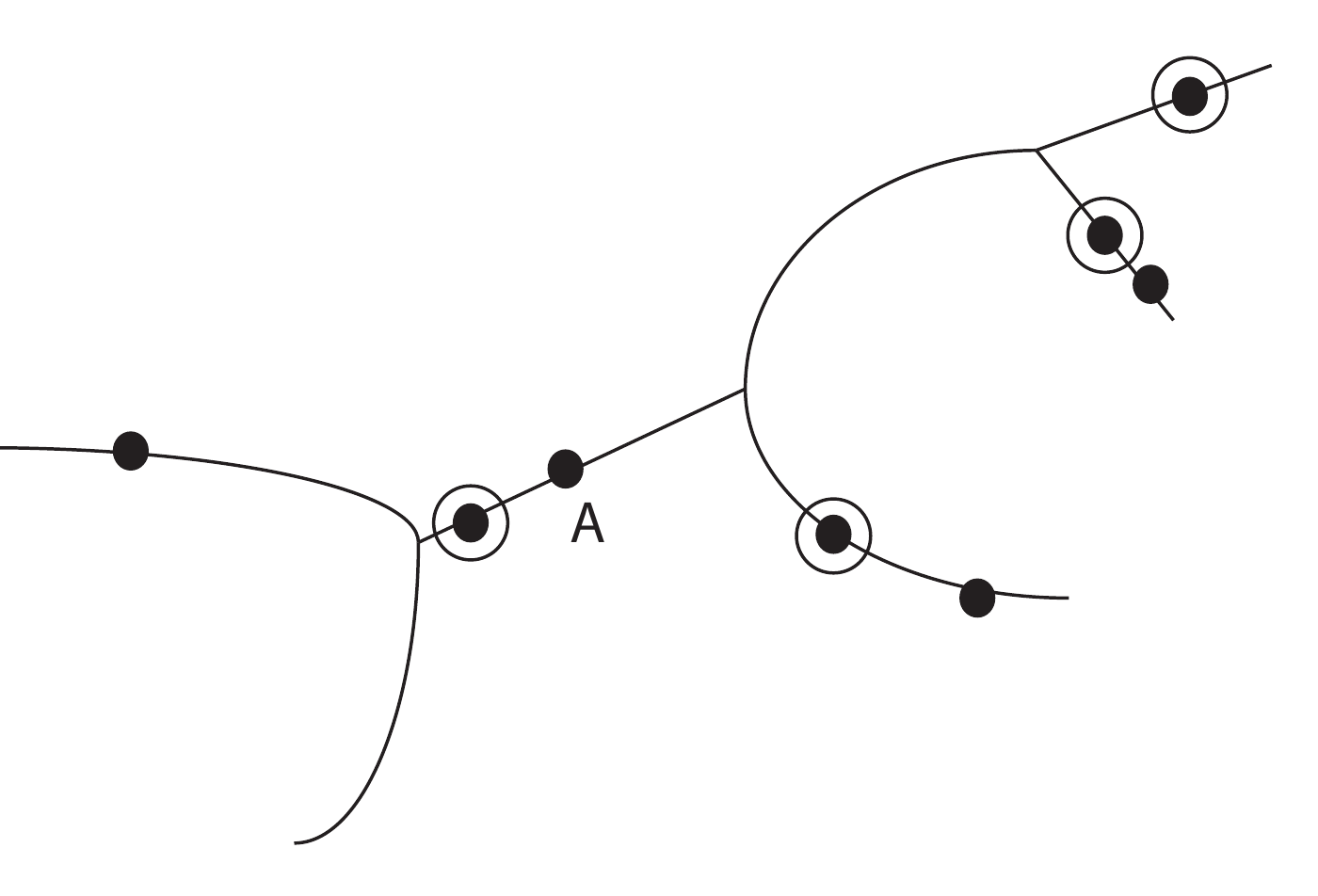}
\caption{Each point represents the center of a maximal disc on $M(G)$.  The neighbors of the disc $A$ are circled.}\label{fig:neighbors}
\end{center}\end{figure*}

\begin{proof}

Assume that the centers of three maximal discs $A$, $B$, and $C$, are on a path of $M(G)$.  Let their centers be denoted by $a$, $b$, and $c$.  The edges of discs $A$, $B$, and $C$ are simply the circles with centers at  $a$, $b$, and $c$ of the same radii as the discs.  Assuming $M(G)$ has no loops, then the path is the only path in $M(G)$ connecting the center of $A$ to $C$.  Suppose that there is an intersection between maximal disc $A$ and $C$ as constructed in Fig.~\ref{fig:abcoverlap}.  Two circles intersect in a region shaped like an asymmetric lens.  Label the two points where the edge of disc $A$ and $C$ intersect $I_{far}$ and $I_{close}$.  Divide the plane into four quadrants defined by the line connecting $a$ and $c$ and the line connecting $I_{far}$ and $I_{close}$.  Now construct disc $B$.  Without loss of generality, we assume the center $b$ is in the bottom right (Quadrant 4) of the figure.  As disc $B$ is constructed, there are restrictions on both where its center can be placed within Quadrant 4 and on its radius.  

First, the point $I_{far}$ must be contained in the disc B.  We observe that as one is traversing the boundary $S$ of the shape, the edge of $A$, $B$, and $C$  must be encountered in the following order: a continuous set of disc $A$ edge points, a continuous set of disc $B$ edge points, a continuous set of disc $C$ edge points, another continuous set of disc $C$ edge points, another continuous set of disc $B$ edge points, and  another continuous set of disc $A$ edge points.   Other continuous sets of other disc edge points may interleave the sets specified, however, the specified order of encountering continuous sets of points of $A$, $B$, and $C$ must still be followed.  If the radius function along the medial axis path is redefined to exactly trace the boundary of $A$, $B$, and $C$, then this ordering must still hold.  Therefore, the intersection points between the edge of disc $A$ and disc $C$ must be contained in disc $B$. 

Second, the radius of disc $B$ cannot be larger than the distance between point $b$ and the farthest point on the edge of disc $A$.  Otherwise all of disc $A$ will be inside disc $B$, making disc $A$ not a maximal disc.  

Third, the points on the edge of disc $A$ not contained in disc $B$ must include points not part of the edge of the asymmetric lens, or else disc $A$ will be completely interior to disc $B$ and disc $C$ and not a maximal disc.

The remainder of the argument reduces to the following.  Given a point $b$ in Quadrant 4, what points on the edge of disc $A$ are farther from point $b$ than $I_{far}$?  If a line is drawn connecting point $b$ to point $a$ then the line intersects the edge of disc $A$ in two locations, one in Quadrant 1 or 2, the other in Quadrant 3 or 4.  The intersected edge point in Quadrant 1 or 2 is the farthest point on the edge of disc $A$ to $b$.  Tracing around the edge of disc $A$ to the other intersection point, each point encountered is closer to point $b$ than the last. Consider the region above a line intersecting $I_{far}$ and point $a$ in Quadrant 4.  If point $b$ was in this region, then only points on the intersection lens edge of disc $A$ are farther from point $b$ than $I_{far}$  This violates the third rule above, so this region cannot contain $b$. If point $b$ is restricted to the remaining region of Quadrant 4, then the distance to $I_{far}$ is alway greater than any other point on the lens.  Therefore, it is the case that the entire intersection overlap region of disc $A$ and disc $C$ is contained in disc $B$.
\end{proof}

The following theorem immediately derives from Theorem \ref{thm:neighbors},

\begin{theorem} Let $R_N$ be a set of maximal discs on $M(G)$.
Let $M(G)$ be divided into two loci of connected points $P_1$ and $P_2$ such that the only points $P_1$ and $P_2$ have in common are a finite set of points occupied by maximal discs $R_{N,boundary}\subset R_N$.  Let $R_{N,1}$ be the maximal discs whose centers are on $P_1$ but do not include $R_{N,boundary}$.  Likewise, Let $R_{N,2}$ be the maximal discs whose centers are on $P_2$ but do not include $R_{N,boundary}$.  The area covered only by the set of discs in $R_{N,1}$ is the same for all possible $R_{N,2}$, and the area covered by only by the set of discs in $R_{N,1}$ is the same for all possible $R_{N,2}$.  \label{thm:isolated}
\end{theorem}

This theorem is illustrated in Figure \ref{fig:parts}.

\begin{figure*}[htb]
\begin{center}
\includegraphics[width=0.75\textwidth]{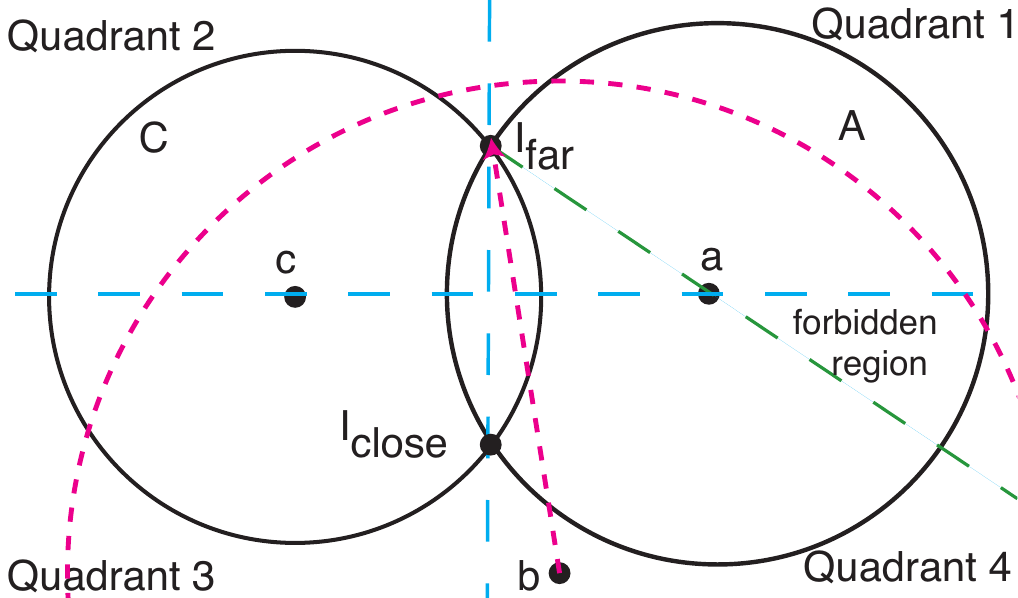}
\caption{A construction to show that disc B must contain all of the overlap between disc A and disc C.} \label{fig:abcoverlap}
\end{center}\end{figure*}

\begin{figure*}[htb]
\begin{center}
\includegraphics[width=0.75\textwidth]{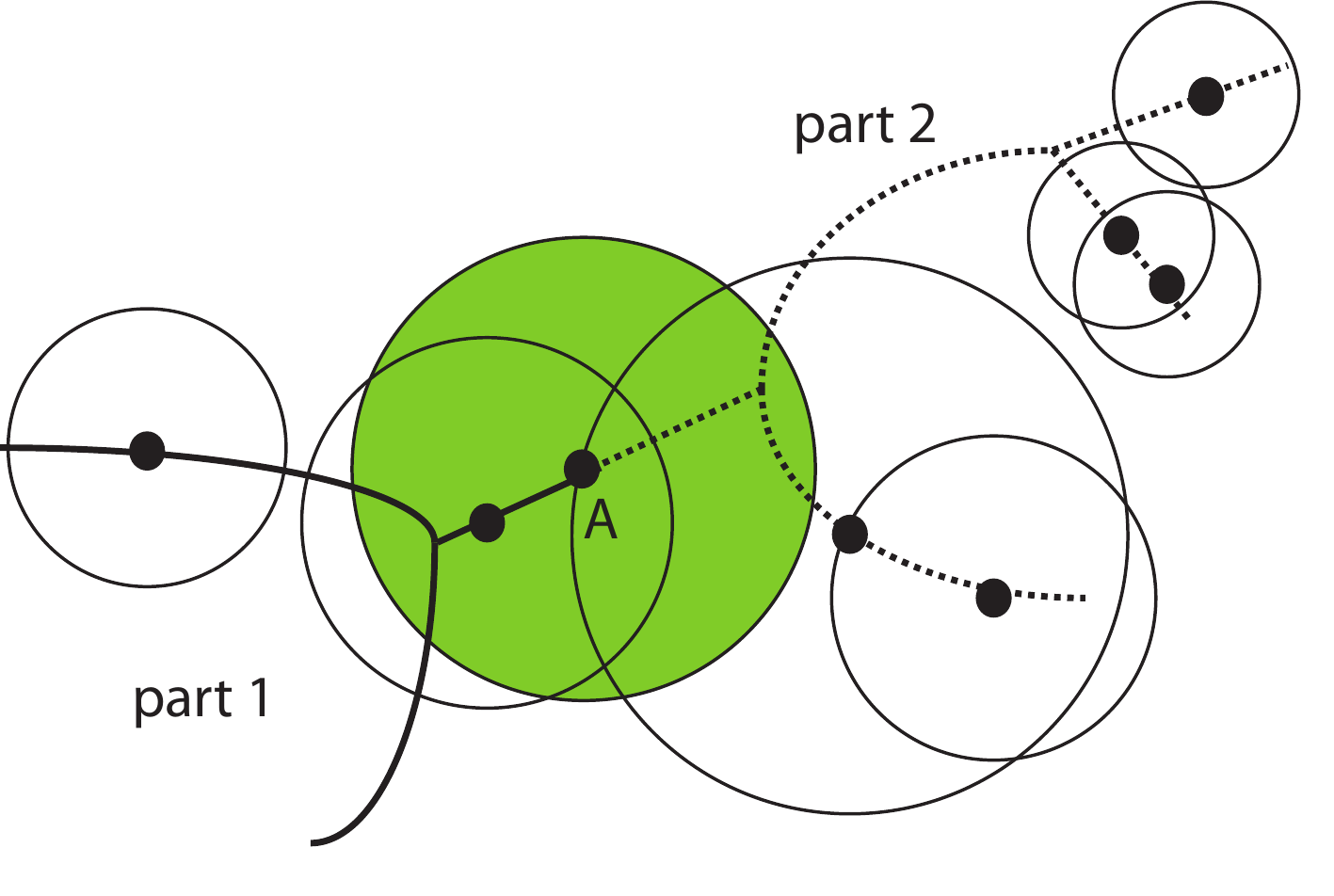}
\caption{If disc $A$ is kept fixed in position, then it divides $M(G)$ from Fig.~\ref{fig:neighbors} into two parts, indicated as solid and dashed lines.  There is no intersection between discs on the two parts of $M(G)$ that is not covered completely by disc $A$ per Theorem \ref{thm:neighbors}.  The two parts, therefore, act as independent spaces. \label{fig:parts}}
\end{center}\end{figure*}

\subsection{Properties of an optimal  planar filling} \label{trapproperties}
We make some observations about the function $\phi(R_N, G)$, where the discs of $R_N$ have centers at points $\bf{x}_i \in M(G)$.  

Per Eqn. \ref{eqn:totalphi}, $\phi$ is a function of the area of discs and the area of overlap between discs.  As  the radius function is continuous over $M(G)$ and the area of overlap between shapes is continuous with respect to inflating or translating the shape, $\phi$ is a continuous function.   

The change in $\phi$ due to moving a single disc center is equal to the change in the area uniquely covered by the disc.  This uniquely covered area can be expressed as the area of the disc minus the overlap between the disc and its neighbors, $O_n(x, r)$, where $x$ is the position of the center of the disc $D$ and $r$ is the radius of the disc.  If the disc center is moved along a path parametrized by $t$, then
 \begin{equation}
 \frac{\partial\phi}{\partial t} = \frac{\partial \mathcal{M}_V\left(V'_1(D)\right)}{\partial t}  = \frac{1}{A_G}\left(  2\pi r\frac{\partial r}{\partial t} - \left(\frac{\partial O_n(x, r)}{\partial x} \frac{\partial x}{\partial t} + \frac{\partial O_n(x, r)}{\partial r}\frac{\partial r}{\partial t} \right)\right)   .
 \end{equation}
 
The function $\frac{\partial\phi}{\partial t}$ is discontinuous when $\frac{\partial r}{\partial t}$,  $\frac{\partial x}{\partial t}$, $\frac{\partial O_n(x, r)}{\partial x}$, or $\frac{\partial O_n(x, r)}{\partial r} $  is discontinuous.   The radius function can be first-order discontinuous at a finite number of points, (e.g. branch points) as can $x(t)$ (e.g. branch points or any point of infinite curvature).   At these points $\phi(t)$ can also be first-order discontinuous.  

If a point of first-order discontinuity is also a local maximum (i.e. $\frac{\partial\phi}{\partial t}$ changes in sign at the point), then small displacements of the neighbors may shift the sided values of the discontinuity without affecting the sign change or shifting the position of the maxima.  The point of first-order discontinuity creates a \emph{center trap}.  The local maximum at a center trap tends to stay stationary unless there are large rearrangements of the points in its neighborhood. As a result, unlike other points on $M(G)$, a center trap tends to be a commonly occupied point in locally maximal filling solutions, optimal filling solutions, and even a fixed feature of filling solutions when $N$ is large.  

A point of infinite curvature in $M(G)$ that is not a branch point, (i.e. its maximal disc is not in contact with the boundary  three or more separate but contiguous sets of points) is an end-point with connectivity two.  

Another point of first-order discontinuity is the point where one disc first contacts another.  These are the points where $\frac{\partial O_n(x, r)}{\partial x}$ or $\frac{\partial O_n(x, r)}{\partial r} $ is discontinuous.  Here the overlap area with another disc changes continuously from zero to positive.  As a sign change cannot occur at this point, it cannot be an isolated local maximum and does not act as a trap.  

We now introduce a new term, \emph{junction point}.
\begin{definition}
A junction point  is a point in $M(G)$ where, relative to some path in $M(G)$ that includes the point, either the radius function or the path itself is first-order discontinuous, or both.  Junctions can act as center traps.
\end{definition}

 Therefore, $\phi(t)$, is a continuous,  piece-wise first-order continuous function. Insights into the structure of $\phi$ and $M(G)$ permit the design of an efficient heuristic with a high likelihood for finding an optimal filling solution.

\subsection{Polygons}
For a convex polygon, $M(G)$ is composed of only line segments.  For a simple polygon, $M(G)$ is composed of line segments and parabolic curves (Figure~\ref{MA}).    For a convex polygon, the parents are always two straight edges.  For a simple polygon, parents can be two straight edges, a straight edge and a reflex point, or two reflex points.   The resultant branches of $M(G)$ and corresponding radius functions are given by the following three cases.

\begin{figure*}[htb]
\begin{center}
\includegraphics[width=0.9\textwidth]{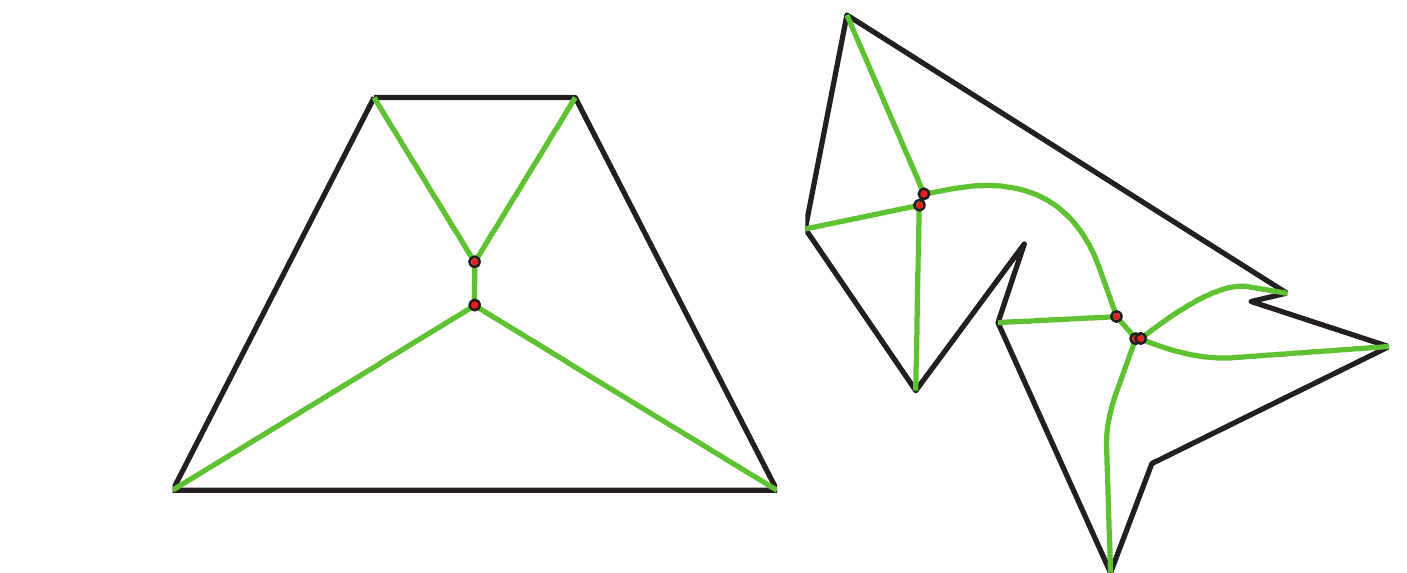}
\caption{Shown are the M(G) of two simple polygons (green online).  The dots (red online) represent junction points.} \label{MA}
\end{center}\end{figure*}

\begin{figure*}
\begin{center}
\includegraphics[width=0.75\textwidth]{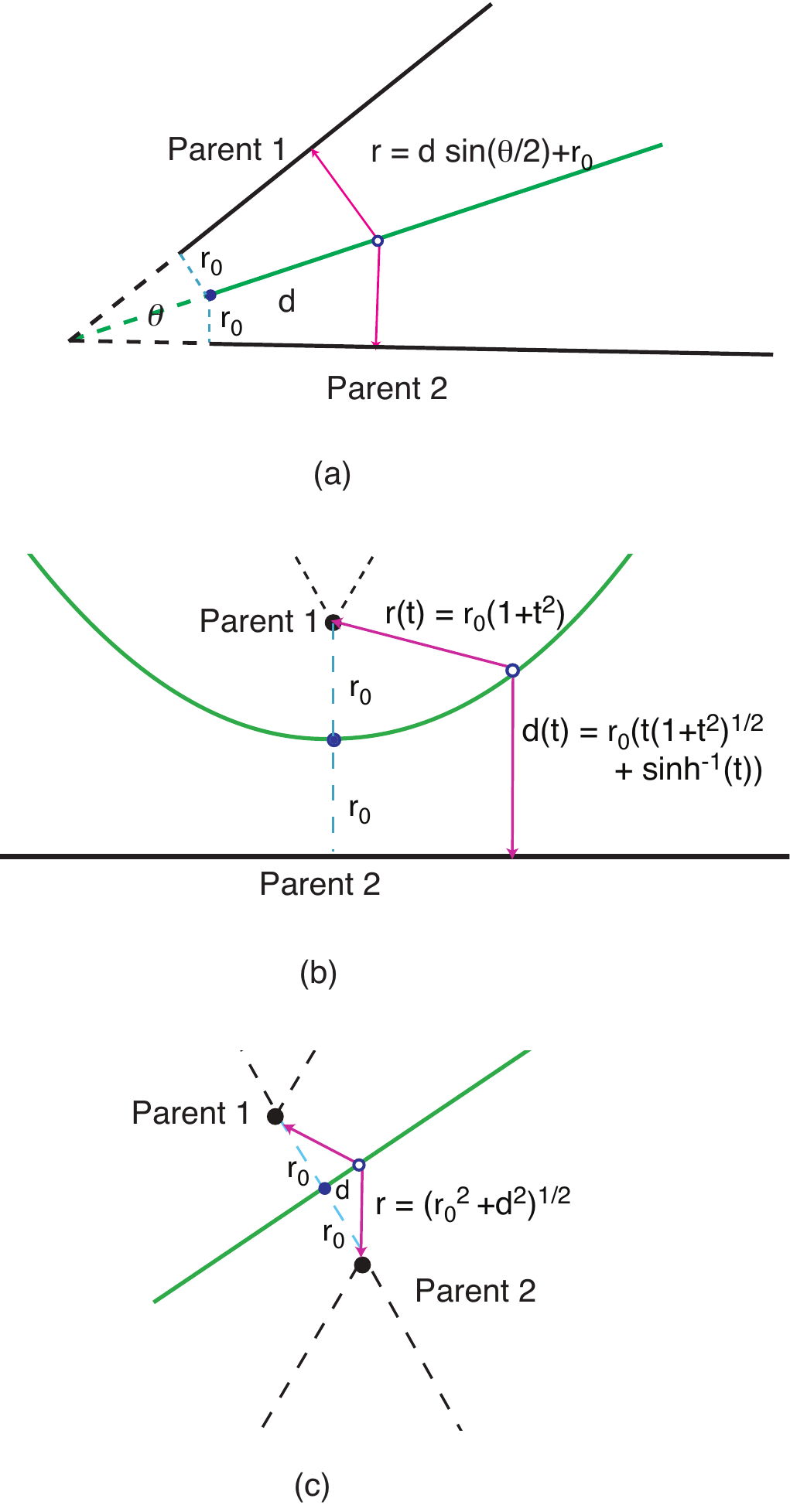}\caption{The three types of branches (green) in a polygon are determined by the parents of the branch.  The parents are (a) two edges (Case 1), (b) an edge and a reflex point (Case 2), and (c) two reflex points (Case 3).  For each branch, the minimum value of the radius function is $a$ and the point is represented with a closed dot.   The open circle represents a disc center at a distance $d$ along the branch.  The radius function in terms of $r_0$, $d$, or a parametrization is provided. } \label{fig:cases}
\end{center}
\end{figure*}

\begin{itemize}
\item{ \emph{Case 1 (Figure (\ref{fig:cases}(a)): two straight edges}}  A line segment with a linear (or constant) radius function.  The path and radius function can be parameterized as $(x(t),y(t)) = \vec{A}t +\vec{B}$,  and $r(t) = ct + r_0$, for $t \geq 0$, for constants $c,r_0 \geq 0$.

\item{ \emph{Case 2  Figure (\ref{fig:cases}(b)): a straight edge and a reflex point}}  A parabolic curve with a non-linear radius function.  The path and the radius function can be parameterized as $(x(t),y(t)) = (2r_0t,r_0t^2)$ and $r(t) = r_0(t^2 + 1)$, where $r_0$ is the minimum of the radius function.  The curvature can be parameterized as $\kappa(t) = \left(1/{2r_0}\right) \left(1+t^2\right)^{-3/2}$.

\item{ \emph{Case 3  (Figure \ref{fig:cases}(c)): two reflex points}} A line segment with a non-linear radius function.  The path can again be parameterized as $(x(t),y(t)) = \vec{A}t +\vec{B}$.   If  $t=0$ is the point on the medial axis halfway between the two reflex points, then $r(t) = \sqrt{a^2  + (a_xt^2-b_x)^2 +(a_yt^2-b_y)^2}$, where $a$ is the distance from the halfway point to the reflex point, $\vec{A}=(a_x, a_y)$ and $\vec{B}=(b_x, b_y)$. 
\end{itemize}

For simple polygons, the point where a line segment meets a parabolic curve (or a parabolic curve joins another parabolic curve) is not a branch point, but does involve a change in the geometry of $M(G)~$\cite{Vilaplana}.    In Figures \ref{fig:cases}(a), \ref{fig:cases}(b), and \ref{fig:cases}(c), the three cases are illustrated with the parents labeled and the radius function shown as a function of the path or segment length, $d$, or of parameter $t$.  Only branch points in simple polygons are junction points.

\subsection{Center-occupied junction points and optimal solutions} \label{section:junctions}

\begin{figure*}[h!]
\begin{center}
\includegraphics[width=1.0\textwidth]{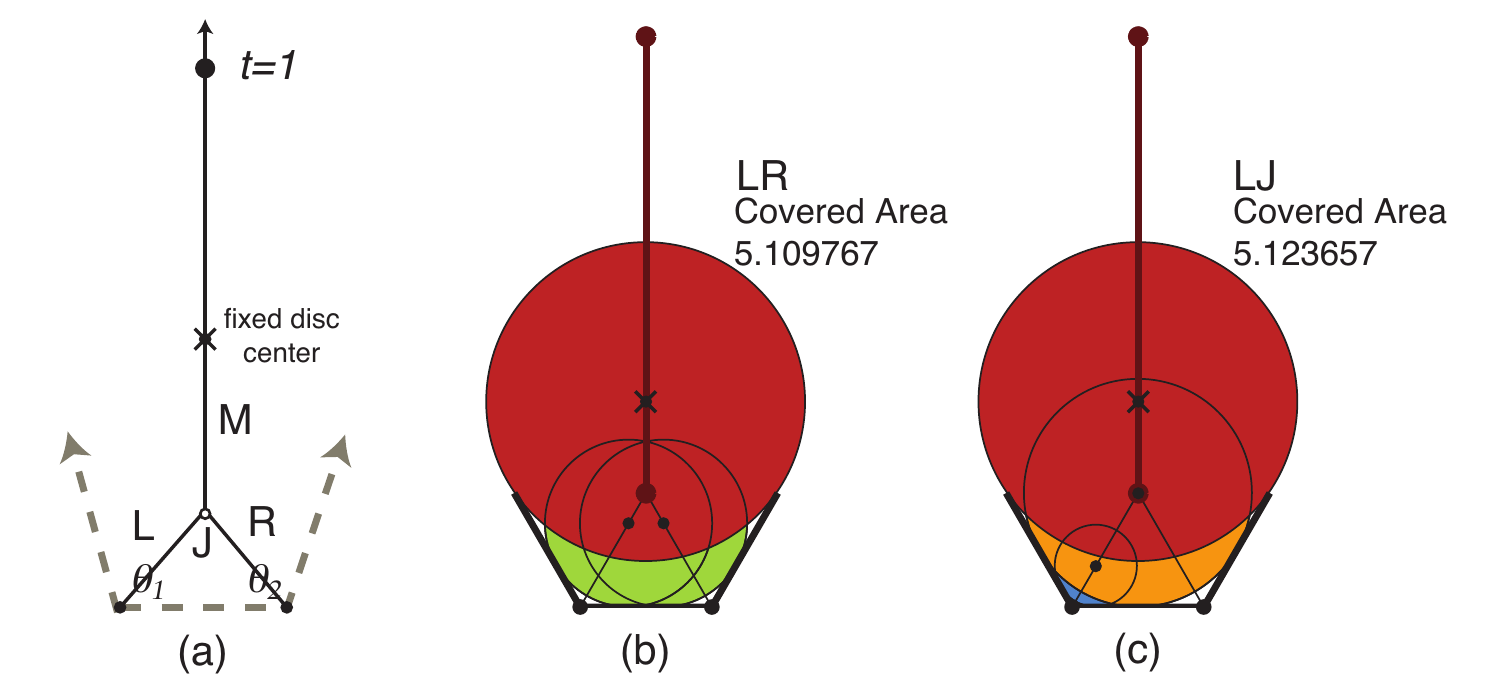}
\caption{In (a) the labeled diagram of the isolated medial axis structure created by three connected polygon edges and a fixed disc is shown.  For $\theta_1 = \theta_2 = 2\pi/3$ and $t = 0.2$, two locally maximal solutions, (b) a symmetrical solution with a disc on the L and R branch and (c) an occupied junction LJ solution are shown.  The second solution is the global maximum. \label{fig:junctionsol}}
\end{center}\end{figure*}

\begin{figure*}[h!]
\begin{center}
\includegraphics[width=1.0\textwidth]{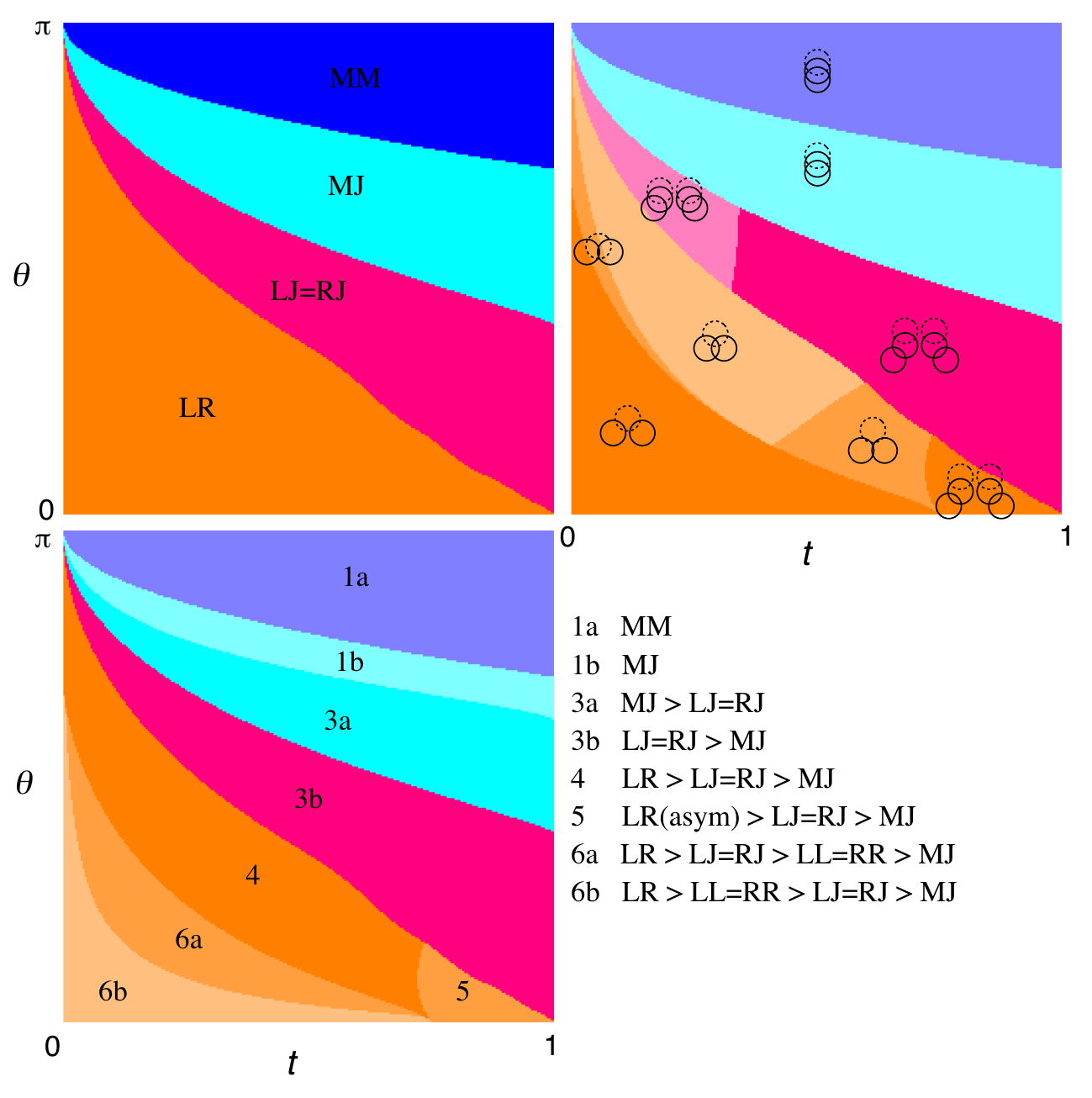}
\caption{Case (A): For the constructed problem of Figure \ref{fig:junctionsol}, $\theta = \theta_1=\theta_2$, on the top left is shown the form of the global maximum as a function of $\theta$ and $t$.  On the right the topological type of the global maximum is shown.  The fixed disc is drawn with a dashed circumference. The two added discs are solid.  On the bottom left, the number and form of the maxima in each part of the phase diagram is indicated.  A key for understanding this diagram is on the bottom right}
\label{fig:sym_nmsol}
\end{center}\end{figure*}

\begin{figure*}[h!]
\begin{center}
\includegraphics[width=1.0\textwidth]{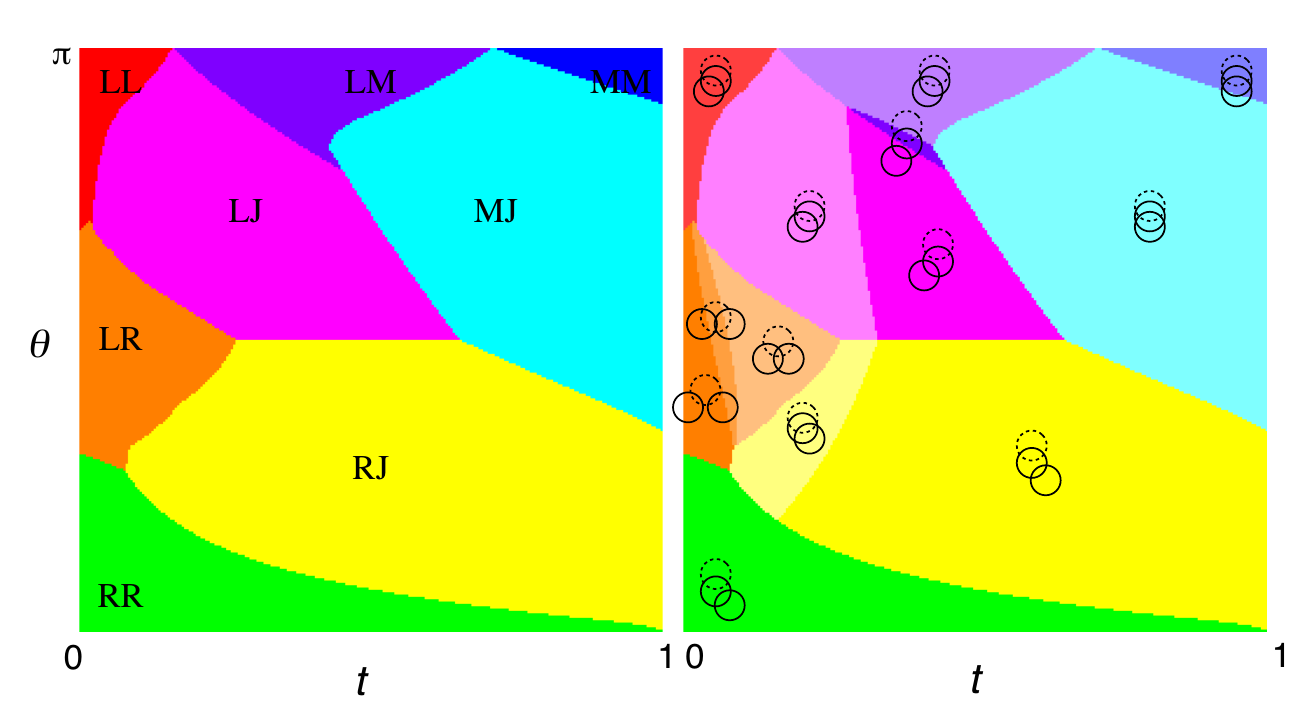}
\caption{Case (B): Diagrams are shown for the constructed problem of Figure \ref{fig:junctionsol}, but with $\theta_1 = \pi/2$,  $\theta_2=\theta$.  On the left is shown the form of the global maximum as a function of $\theta$ and $t$.  On the right the topological type of the global maximum is shown.  The fixed disc is drawn with a dashed circumference.  The two added discs are solid.}
\label{fig:rightangle_nmsol}
\end{center}\end{figure*}
The strategy of the optimization algorithms in the subsequent sections is to find the optimal filling solution, or the global maximum of $\phi$, by generating a population of local maxima.  Junction points are special points that need to be handled separately from branches by a maximum finding search method.    

To demonstrate the diverse landscape of these maxima, and the role of junction points therein, we employ the following numerical experiment.  A section of a polygon $G$ is constructed from a ray, a segment, and another ray having two internal angles, $\theta_1$ and $\theta_2$.   The medial axis of this shape is two straight branches (parents are a ray and the segment) and a straight branch that is a ray (parents are the two rays)  that meet at a junction point, as shown in Figure \ref{fig:junctionsol}(a).     The ray branch is parameterized by defining $t=0$ to be the junction point and defining $t=1$ to be the point on the ray where a maximal disc would be tangent to a disc occupying the junction point.    We fix a disc at the $t$ position on the ray branch for $0\leq t\leq1$ (e.g. the large red discs in Figure \ref{fig:junctionsol}(b) and \ref{fig:junctionsol}(c) with the $\times$ at their centers.).    This region of the polygon below the disc is now completely isolated from the rest of the polygon per Theorem \ref{thm:isolated}.   We now search for all the local maxima that can be constructed by adding two discs below the fixed disc.    We perform this search for two cases, (A) angles $\theta_1$ and $\theta_2$ between each ray and the segment are set identically to $\theta$ for $\theta\in[0,\pi]$, and (B) $\theta_1 = \pi/2$ and $\theta_2 = \theta$ for $\theta\in[0,\pi]$.  In Figure \ref{fig:junctionsol}(b) and \ref{fig:junctionsol}(c), two examples of local maxima are shown for a case (A) construction where $\theta_1=\theta_2= 2\pi/3$ and $t$ = 0.2.  Figure \ref{fig:junctionsol}(a) shows the graph structure defined by the medial axis of this construction, with branches labeled L, R, and M, connected by junction point J.    All local maxima of this construct can be classified by a two letter code representing the branches or the junction point that the disc centers occupy.  We shall refer to the two letter code as the form of the maximum solution. Figure \ref{fig:junctionsol}(b) is a LR maximum (one disc on the L branch, the other on R) and figure \ref{fig:junctionsol}(c) is a LJ maximum (one disc on the L branch, and another on the J junction point.).  In this case the LJ maximum is the global maximum.  A RJ maximum would be symmetrically equivalent, and the global maximum is degenerate.


In the top left of Fig. \ref{fig:sym_nmsol} a diagram showing the forms of the global maximum are shown for case (A) for $\theta$ = 0.0 to $\pi$ radians and $t$ = 0 to 1.  In the region labeled (LJ=RJ), the global maximum is degenerate, as in figure \ref{fig:junctionsol}(c).  We observe that an occupied junction point is part of the global maximum for a large region of the phase diagram (LJ=RJ and MJ).  In the top right of Figure \ref{fig:sym_nmsol} the topological types of the different global maxima regions are shown, each region shaded a different color.  The dotted circles in the diagram represent the fixed disc.  The solid circles represent the other two added circles.  Two global maxima are considered topologically equivalent if they have the same type of edge intersections. The MM and MJ regions, for example, have the same type of topological intersection between the three discs.   The LJ=RJ region has two types (colored with two shades of pink), and the LR region has five (colored with five shades of orange).

In the LR region of the diagram, the form of the global maximum is symmetric (similar to figure \ref{fig:junctionsol}(b)), except at high $t$.  In this region, the disc on the L (R) branch is larger and overlaps the fixed disc, while the disc on the R (L) branch is smaller and overlaps only the other non-fixed disc.  This region, therefore, also has a degenerate global maximum form.  

In the bottom left of Figure \ref{fig:sym_nmsol} the number and forms of local maxima found in each part of the diagram is shown.  The number of local maxima ranges from one to six.  On the bottom right is a key for indicating which forms of local maxima can be found in each region.  For example, in region 6b, there are six local maxima.  The symmetrical LR maximum has a higher filling than a LL maximum, which is symmetrical identical to an RR maximum.    The LL maximum has a higher filling than an LJ maximum which is identical to an RJ maximum.   An MJ maximum has the lowest filling of all the local maxima in this region.   Notably, the MJ local maximum is present everywhere on the phase diagram, except in region 1a, where MM is the single only maximum.  As $\theta$ and $t$ both approach zero, more local maxima emerge.  

On the left of figure \ref{fig:rightangle_nmsol} the form of the global maximum is shown for case (B) where $\theta_1 = \pi/2$ and $\theta_2 = \theta$ = 0.0 to $\pi$ radians and $t$ = 0 to 1.   For case (B), there is no symmetrical solution region (aside from a region of zero area where $\theta_2 = \pi/2$   The space is divided into eight forms of global maxima.   The only global maximum solution form that is not found is the RM solution form.   On the right of Figure \ref{fig:rightangle_nmsol}, the topological type of the global maximum is shown.   The solution regions MM, MJ, RR, and LL each have a single topological type, the solution regions LJ, RM, and RJ have two, and the solution region LJ has three.  Interestingly, the case (B) LR region has two fewer topological forms than the case (A) LR region from figure \ref{fig:sym_nmsol}.   For case (B), a diagram of the number of local maxima as a function of $\theta$ and $t$ is not shown, as we found the landscape too complicated to map.

 The number of maxima and solution forms and topological types depicted in both Figure \ref{fig:sym_nmsol} and \ref{fig:rightangle_nmsol} are surprisingly complicated.  While the boundaries between different solution regions may correspond to some analytical expression, the expression is not known.  If the number of possible forms is small, it is simplest to check each form to find the global maximum.  We observe that the junction point on $M(G)$ is occupied in the global maximum solution for a large fraction of the diagram.  This numerical experiment justifies the treatment of junction points as special points in the medial axis point set.  This numerical experiment also discourages searching for an analytical expression to finding optimal filling solutions for at least small finite sets of discs.

\section{Algorithms for Generating Filling Solutions} \label{sec:algs}
We know of no analytical method for finding the optimal filling solution for a given $N$.  Instead we introduce two algorithms that explore the objective function landscape of $\phi$ to search for the global maximum.
  
\subsection{Genetic Algorithm}  
A genetic algorithm (GA)~\cite{GA} is employed to find the optimal filling solution for polygons.  The benefit of the GA is that it uses a minimal number of mathematical assumptions about the space of the filling solutions, although the computational time required is prohibitively long for $N > 20$.   

\subsubsection{Algorithm Description} GAs start with an initial random population of solutions that are combined and mutated until no better solutions are found after a fixed number of iterations.  

For this implementation, 100$N$ to 400$N$ population members are initialized. Each member is an ordered set of coordinates of N discs. If a disc is randomly generated outside the polygon, it is moved inside.

Without loss of generality, but with dramatic improvement of the efficiency and accuracy, the GA assumes that solutions will consist of maximal discs and attempts to construct them.  First, the radius of each disc is grown to touch the nearest edge in the polygon.  Second, if it can be determined that the disc is in a corner of the polygon  (e.g. the nearby medial axis is a straight branch terminating in an end point), then the disc is moved to the nearest point on the medial axis, constructed by generating the bisector of the corner's internal angle. 
 
These constraints are applied to the entire population and then $\phi$ is computed for each member. The population is then sorted by $\phi$ to produce a list of ranked solutions from best to worst. Members of generation $g$ are randomly chosen as parents for generation $g+1$. The relative probability $p$ of a given member being chosen is weighted by its rank $r$,  $p=1 / \sqrt{r}$.  The next generation of members is created from the current generation as follows.  

\begin{itemize}
\item{\emph{Best}  The best members, unmodified, are included. }
\item{\emph{Mutation}  One ``parent" member is randomly chosen and is randomly mutated by either moving a disc randomly, displacing a disc up to 1/2 the polygon's width, $w$, displacing a disc up to 1/200th of $w$, or moving a disc to a junction point  point.  The mutated ``child'' member is included.  }
\item{\emph{Crossover} Two parent members are selected and their discs are spatially sorted by $x*w+y$, where $(x,y)$ is the position of a disc.   A crossover point $C \in [1,N]$ is randomly determined. The child member contains the discs with indices from 1 to $C$ from the first parent and the discs from $C+1$ to $N$ from the second.  The child member is then included.}
\end{itemize}

The fraction of the next generation created by methods above can be modified to improve the outcome of the algorithm.

Even with the GA's capability of exploring many local maxima to find the global, it can still  get trapped in a local maximum. The GA is run with different random number seeds ten or more times for each shape and value of $N$. The best of the best solutions obtained over all these runs is selected as the GA's final answer.

The mathematical assumptions the GA uses are first, per Theorem \ref{thm:maximal}, that solutions should consist of maximal discs.  Second, the GA also randomly places discs centers on junctions.  As discussed in sections \ref{trapproperties} and \ref{section:junctions}, junction points can act as center traps and are occupied as part of many local maxima.   However, the basin of attraction around the junction point  can be small enough that small displacement mutations do not find it.

\subsection{Heuristic Algorithm for Filling a Polygon}
In this section we introduce a heuristic algorithm that generates a putatively optimal filling solution of $N$ balls, by exploiting the properties of the $M(G)$ structure to generate a collection of unique local maxima.  If enough are generated, the global maximum is among them.
For the two-dimensional filling problem, we propose a local maxima generating strategy whereby centers are distributed onto $M(G)$ and the local filling maxima for that initial guess is found by simple gradient methods (e.g. active set or sequential quadratic programming optimization schemes).  We also propose a method for reducing the number of such distributions needed to find an optimal solution by using the $N-1$ filling solution to generate the $N$ filling solution.

The first step is to intelligently divide up $M(G)$.  The medial axis is divided into $K$ pieces, maximally long branch sections with monotonically increasing radius functions and the junction points connecting them.  To generate these pieces for a polygon, Case (2) and Case (3) may need to be divided into separate sections and joined with other branch sections.  The medial axes of the left and right polygons depicted in Figure \ref{MA} are composed of seven and seventeen pieces, respectively.

\begin{definition}
A \emph{way}, $W$, is a distribution of $N$ discs over the $K$ pieces (branch sections and junctions), $W= \{n_i\}_1^{K}$ where $N=\sum_i^K n_i$  and $n_i \in \mathbb{N}$.  If the $i$-th piece is a junction point  then $n_i \in \{0,1\}.$
\end{definition}

\begin{conjecture}
There is at most one local maximum per \emph{way}. \label{conj:onemax} 
\end{conjecture}

If Conjecture \ref{conj:onemax} holds, then to find the optimal $N$ filling solution, a maximum must be generated for every way of $N$ discs and $K$ pieces.   If $J$ is the number of pieces that are junctions, then the number of maxima to be searched is of order $O\left(N^{K-J-1}\right)$ (see A-4).

\begin{conjecture} Given the optimal \emph{way} of distributing $N-1$ discs, $\{n'_i\}_1^{K}$, the optimal \emph{way} of distributing $N$ discs is nearby, where nearby means $\sum_1^K \mid n_i - n'_i \mid$ is small, and that if the discs assigned to a given piece is decreased, the pieces that have discs increased have a minimal distance (counted by number of connecting pieces) to the decreased piece. \label{conj:nearby}
\end{conjecture}

\begin{conjecture}
For a given $G$ and $M(G)$, there is an $N'$ such that for $N \geq N'$, the junction points are always occupied in the optimal filling solutions.\label{conj:junctions} 
\end{conjecture}

Given the \emph{way} of the $N-1$ filling solution, the heuristic generates the local maxima of the nearby \emph{ways} using a local maximum finding technique.  The best local maximum found is presumed to be the optimal $N$ filling solution for the shape.  This heuristic is made more efficient by taking advantage of center occupied junction points and the dependence of the filling function on the nearest neighbors.  We implement this heuristic for polygons, which have a limited set of medial axis parameterized pieces to be considered. 

 \subsubsection{Detailed Description of Heuristic}

 Following is a more detailed description of the Heuristic Algorithm (HA).  
\paragraph{Auxillary Algorithms}
The following sub-algorithms are needed to deploy the HA.
\begin{enumerate}
\item{\emph{Generating and Dividing $M(G)$ into $K$ pieces.}}
The medial axis of the polygon is generated\footnote{using the matlab software package MatlabMedialAxis-Version 2.0 provided by Suresh Krishnan~\cite{Suresh}}.   Parabolic curves (Case 2) and straight curves (Case 3) that  include a minimum in the radius function are split at the minimum.  The split branches are then recombined to form maximally long paths with monotonically increasing radius functions.  Branch points are separated from branch pieces as junction point  pieces. 

\item{\emph{Calculating the Area of a Union of Discs.}}  The total area of the union of the discs is determined analytically  by dividing the space into intersection regions defined by boundary arcs and calculating the area of each region~\cite{Vakulenko}.   The method can be applied to the entire set of discs, or, far more efficiently, by dividing the calculation over the discs on each piece.  In this latter method, first the area of the union of discs on each piece is calculated and summed.  This sum over-counts the overlaps between unions of discs of different pieces.  Second, the overlap between the disc at the end of a piece and its neighboring discs on other pieces is subtracted from the total, once for each time it was over-counted.  This latter method is more complicated, but also more computationally efficient because the areas of smaller sets of discs are calculated.

\item{\emph{Partitioning a Graph by Occupied Junctions.}}  By taking advantage of regions of the $M(G)$ graph isolated by occupied junctions, per Theorem \ref{thm:isolated}, filling solutions can be divided into solutions of independent sub-spaces of $M(G)$.  Using the topology of the $M(G)$ graph and a set of maximal discs $R_N$, this algorithm step divides the graph into \emph{parts}, or sets of pieces isolated from each other by occupied junctions.

\item{\emph{Finding the Local Maxima.}}  A solution set of discs can be uniquely represented by the way $W= \{n_i\}_1^{K}$ and a set of parameters $\{t_{i,j}\}$ where $i \in [1,n_i]$ and $j\in [1,K]$ and $t_{i,j} \in [0,1]$\footnote{if a piece terminates in a junction point  at $t$=0 or 1 or both, then  $t_{i,j} \in (0,1]$,   $t_{i,j} \in [0,1)$, or  $t_{i,j} \in (0,1)$, respectively.}  Given an initial guess way and a parameter set, an optimization method (e.g. active set or sequential quadratic programming optimization schemes) is applied by using $\phi$ as the objective function.  If the initial guess includes a trial disc insertion into a piece of a given part of $M(G)$, then all $t_{i,j}$ parameters of the part are free parameters.  All $t_{i,j}$ parameters outside the part are fixed.

\end{enumerate}

\paragraph{Heuristic Algorithm}

Given the $N-1$ solution:
\begin{enumerate}
\item{For each part of the $M(G)$ graph isolated by occupied junctions, a new disc is inserted into each piece and the best solution for the part is found. }
\item{For each occupied junction point  of the $N-1$ solution, the disc center is removed from the junction point  and initial guesses are generated by inserting two discs into nearby pieces and finding the local maximum.  The more combinations of nearby pieces are tried, the larger a neighborhood is considered.}
\item{The $N$ filling solution is constructed from the best trial solution found.  If a piece $k$ has a parameter value $t = 0$ or 1, indicating that the junction at the end of the piece has been occupied,  then the disc is moved from piece $k$ to the junction point  piece.  If a solution was generated for a part of $M(G)$ and not included in the best solution, and the part is found in both the $N$ solution and the $N-1$ solution, then the solution is cached.}

\end{enumerate}

\subsubsection{Algorithmic Efficiency}

\begin{figure*}[htb]
\begin{center}
\includegraphics[width=1.0\textwidth]{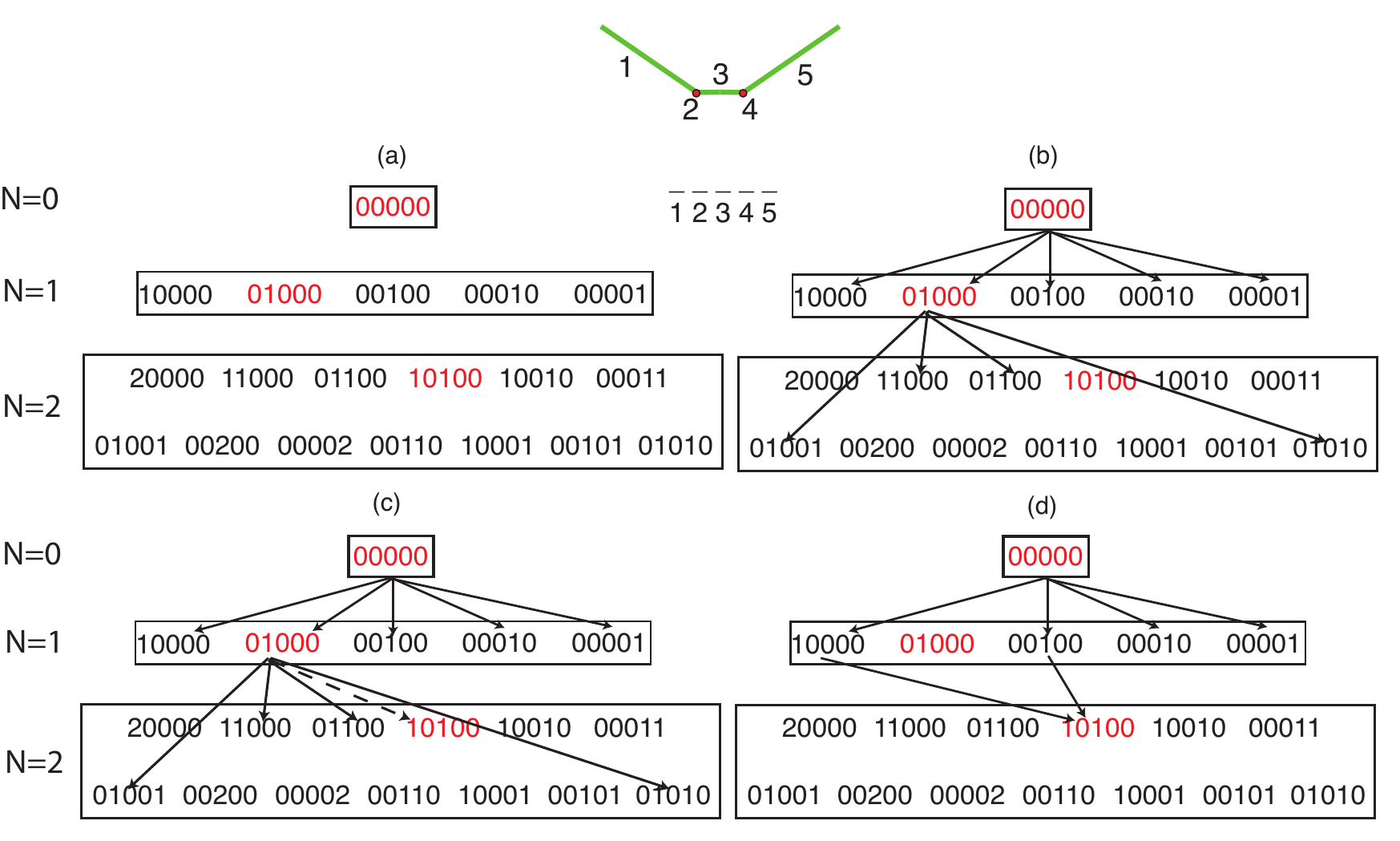}
\caption{At the top of the figure is a medial axis with five pieces, three branches and two junctions.  (a) The full table of \emph{ways} is shown for N=0, 1, and 2.  In (b) the search space is reduced using the greedy assumption that the next best solution is related to the last best solution.  (c)  We also add searches that deoccupy junction points and inserts discs onto nearby branches.  (d)  If the best 1-way was not searched, two of the four remaining 1-ways would have searched the best 2-way on the next iteration.  \label{fig:efficiency}}
\end{center}\end{figure*}

\begin{figure*}[htb]
\begin{center}
\includegraphics[width=0.6\textwidth]{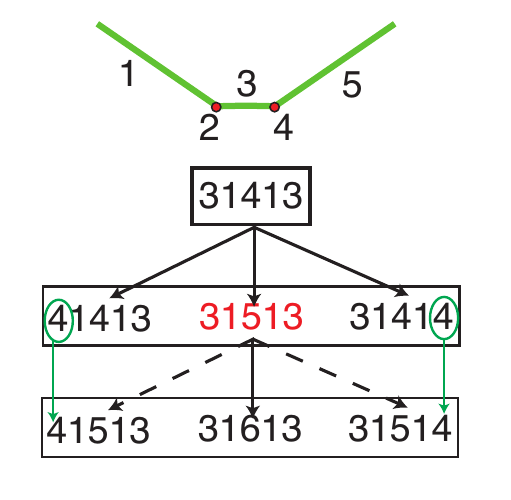}
\caption{ Assume that the junction points 2 and 4 stay occupied.  To generate the $N$=13 solution, three ways are searched for a local maximum.  To generate $N$=14, only one additional way, 31613, needs to be searched.  The ways 41513 and 31514, can be created by combining the search of branch 1 and 5 with the solution of branch $3$ for $N$=13.  The occupied junction points isolate the solutions on each branch from solutions on the rest of the medial axis. \label{fig:caching}}
\end{center}\end{figure*}

\paragraph{Using Greediness to Reduce the Search Space} \label{greed}  The heuristic exploits Conjecture \ref{conj:nearby} to reduce the number of ways to search for local maxima (i.e. number of initial guesses)  from $O\left(N^{K-J-1}\right)$ to $O(N(K+J))$.   Figure \ref{fig:efficiency} depicts a hypothetical medial axis with two junction points and three branch pieces.  Figure \ref{fig:efficiency}(a) shows a table of all possible ways for $N$=0, 1, and 2.  Rather than search each way of $N=2$, a reduced set is searched.  That set is generated as follows, given the best $N-1$ way, one disc is added to each piece (that is not an occupied junction) as shown in Figure \ref{fig:efficiency}(b).  Then, for each occupied junction point  in the $N-1$ best way, the junction point  is deoccupied and disc is inserted into two of the branches nearby the junction, as shown in Figure \ref{fig:efficiency}(c).   The maximum number of ways that will be searched, given the best found $N-1$ way, is $K + AJ$, where $A$ is a constant dependent on how large a neighborhood of a junction point  one chooses.   While this heuristic is not guaranteed to find optimal solutions, it finds a putatively optimal $N$ filling solution with only $O(N(K+J))$ number of searches.

For example, for a triangle with $K$ = 4, finding the best arrangement of $N=10$ discs means searching 121 ways, and the best arrangement of $N=100$ discs requires searching 10,201 ways.  By using a heuristic that exploits Conjecture \ref{conj:nearby}, to  finding the best arrangement of $N=10$ discs requires searching only 70 ways, and for $N=100$ discs, only 700 ways.

\paragraph{Applying optimization techniques to $N'$ discs where $N' < N$} \label{fewdiscs}

The computational effort to calculate the analytically exact area of the union of a set of $N$ discs is super-linear in $N$, as also can be optimization methods of $N$ parameters.  The exact order of the computational effort is dependent on the arrangements of the discs and the details and convergence rate of the optimization algorithm.  The greedy heuristic of section \ref{greed} not only requires searching fewer ways, but also mostly searches ways of $N'$ discs where $N' < N$.  This significantly improves the computational efficiency of finding a solution.

\paragraph{Efficiently Sub-Dividing the Search Space} \label{divide}
The heuristic also improves efficiency by exploiting the properties of the solution space per Theorems \ref{thm:neighbors} and \ref{thm:isolated} and the behavior of center traps as discussed in section \ref{trapproperties}.

When a junction point  is occupied by a disc center in the $N-1$ solution, the center is usually trapped and the phase space of centers can be divided into independent sub-spaces.  If it is known (or guessed) that the best solution for $N$ also includes a center at the junction, then the sub-parts of $M(G)$ connected only by the junction point can be searched independently.  Per Theorem \ref{thm:isolated}, rearrangements of centers in one sub-part cannot affect the best arrangement of centers in another if they are connected only by a center-occupied junction.  

This means that searches can be performed on a subset of the $N$ discs (efficient per section \ref{fewdiscs}) and that solutions of independent sub-spaces of $M(G)$ can be cached.  Figure \ref{fig:caching} shows how, when junction points are presumed to remain occupied, only one additional search of a way is needed to generate the next putatively optimal $N$ filling solution.  Per Conjecture \ref{conj:junctions}, at sufficiently large $N$, junction points are occupied.  This implies that for large $N$, generating the optimal $N$ filling solution from the optimal $N-1$ filling solution requires a search of only one additional way, reducing the complexity of the HA to $O(N)$ searches. 

\subsubsection{Self-Correcting}
When implementing the HA, a practical choice is made as to how large a neighborhood of ways $N$ discs that are nearby the optimal $N-1$ way will be searched.  There is a computational trade-off between searching only ways such that $\sum_1^K \mid n_i - n'_i \mid=1$, in which case the optimal way may be missed, or such that  $\sum_1^K \mid n_i - n'_i \mid < \infty$, in which case the optimal way can not be missed but the search space has not been reduced.  

One weakness of a method that uses the $N-1$ solution to find the $N$ solution is that, if the optimal $N'$ solution is not found, all solutions for $N > N'$ may not be optimal as well.  However, we observe that in most cases where the HA does not find the optimal $N'$ filling solution, by some $N>N'$, the HA is generating the optimal solution again.   That is, even if the wrong solution is found, we observe that the solution finding method tends to self-correct at a higher $N$.  In Figure \ref{fig:efficiency}(d), for example, if the optimal $N$=1 way was omitted from the search, the optimal $N$=2 way would still be searched by two of four alternate $N$=1 ways.   

Presuming Conjecture \ref{conj:junctions} is true, we can show that for at least one common case, an HA searching only a small neighborhood of ways still always self-corrects.  Consider, an $M(G)$ with $K$ pieces and $J$ junctions, where all pieces that are sections of branches are connected by pieces that are junctions (e.g. G is a convex polygon).  Assume that we restrict our search to ways such that all the $J$ junctions are occupied for $N \geq J$.  The $N-J$ remaining discs will be partitioned over the remaining $K-J$ pieces.  For this case, per Theorem \ref{thm:isolated}, the increase in the filling measure due to adding an arbitrary number of discs to a piece can be solved independently.   Also, adding discs sequentially and optimally to a given piece strictly increases $\phi$, but the change in $\phi$ monotonically decreases.  It follows that the best $N$ way found with the given constraints can never include removing a disc from a piece.  Thus the heuristic will search for the best $N$ way by comparing the local filling maxima generated by adding one disc to each of the $K-J$ pieces.  It follows that the heuristic will always find the best $N$ filling solution for which the $J$ junctions are occupied.   If Conjecture \ref{conj:junctions}, is true, since for some $N>N'$ all the junctions will be occupied in optimal solutions, it follows that despite having not generated the optimal way for all $N \leq N'$, the heuristic generates the optimal way for all $N>N'$.

We propose a stronger conjecture than Conjecture \ref{conj:junctions}.

\begin{conjecture}
For a given $G$ and $M(G)$, there is an $N'$, such that for $N \geq N'$, the junction points are always occupied in all filling solutions that are local maxima.\label{conj:junctionsstronger} 
\end{conjecture}

If this conjecture is true, then it would also follow for $M(G)$ with $K$ pieces and $J$ junctions, where all pieces that are sections of branches are connected by pieces that are junctions, the heuristic will always self-correct and generate the optimal way for sufficiently large N.

\subsubsection{When the Heuristic Algorithm fails} 

Even if the Conjectures \ref{conj:onemax}, \ref{conj:nearby}, and \ref{conj:junctions} above hold, in practice this HA may still fail to find the optimal solution for the following reasons. 

\paragraph{(1) Assuming a \emph{way} has no local maximum}  While searching for the local maximum associated with a \emph{way}, it is common to generate the local maximum of a nearby \emph{way} instead (e.g a junction point  becomes occupied).  This leads to the conclusion that the \emph{way} has no local maximum.   However, the search may simply have been initiated outside the basin of attraction of the local maximum of the \emph{way}.    

\paragraph{(2) Searching in too small a neighborhood}  As discussed above, some optimal $N$ solutions require looking in a larger neighborhood of the $N-1$ solution. Tradeoffs that balance confidence in finding the optimal solution against the computational cost of searching larger neighborhoods may result in optimal solutions being missed. 

\paragraph{(3) Solutions are only as good as the optimization method applied}  Lastly, local maxima finding techniques can have trouble converging.  This is not a failure of the Heuristic Algorithm, per se, but occasionally affects the HA solution.  Switching which nonlinear constrained minimization optimization technique is being applied generally solves the problem.
 
\subsection{Heuristic vs. Genetic Algorithm Filling Solutions}
To assess the capability of the heuristic algorithm \emph{vs.} the genetic algorithm, solutions were generated for $N$=1 to 21 for a selection of five convex polygons and 21 concave polygons.  The putative best solutions produced by this HA match well the solutions generated by the GA.  Specifically, the HA almost always produces solutions of the same \emph{way} as the GA.  The gradient optimization technique employed by the HA is usually better at converging to a final set of disc positions for a given \emph{way} than the GA.   On rare occasions the HA and GA find different \emph{ways}.   When the HA \emph{way} is better, the GA has usually become trapped in the wrong local maximum.  When the GA \emph{way} is a better solution, we find that the \emph{way} was outside the neighborhood that was searched by the HA.
Examples of filling solutions are shown in Figure \ref{fig:fillingexamples}.
\\ \\
\begin{table} 
\begin{tabular}{|c|ccc|c|}
	\hline
 &  HA and GA  &   Best Way:  & Best Way:   &  Best $\phi$ :  \\
 & Way Match & HA & GA & HA \\
	\hline
Convex   & 98.1\%  & 1.9\% & 0\%  & 100\% \\
Concave   & 92.97\% & 3.4\% & 3.63\%  & 96.37\%\\
	\hline
\end{tabular}
\caption{Table 1.  A comparison of the filling solutions generated by the HA and GA for five convex polygons and 21 concave polygons for $N$=1-21.}
\end{table}

\begin{figure*}
\includegraphics[width=1.0\textwidth]{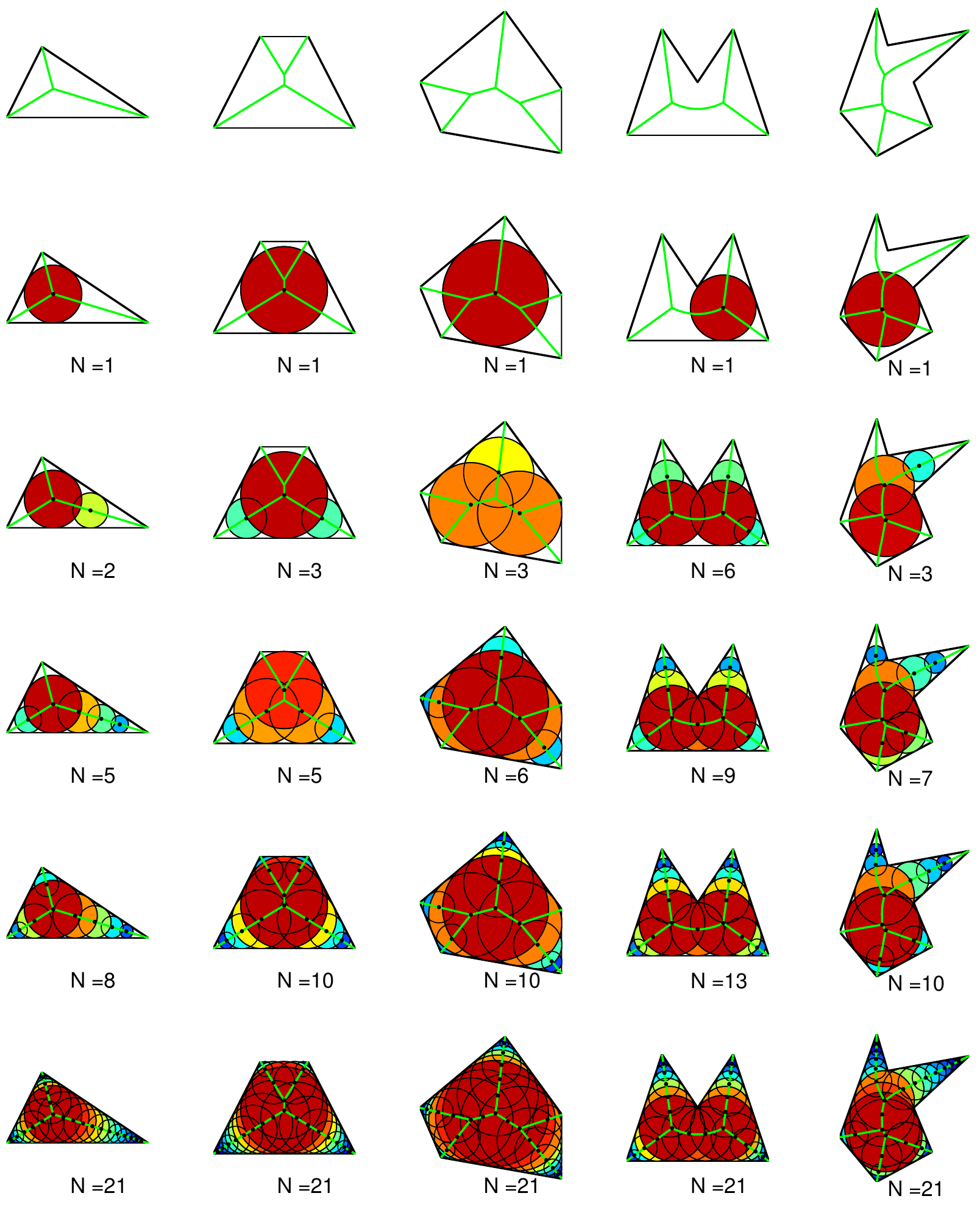}\caption{ Examples of the optimal filling solutions of three convex and two concave polygons for $N$=1-21.  The top row shows the medial axis of each polygon.\label{fig:fillingexamples}}
\end{figure*}

\section{Optimally filling a polygon as N $\rightarrow \infty$} \label{sec:contlim}
It is instructive to examine how the optimal filling of a shape converges to the total volume of the shape as $N\rightarrow \infty$.   As $N\rightarrow \infty$, the centers will be distributed densely in $M(G)$ such that the change in the density of centers measured over small intervals of the 1-manifolds $M(G)$ can be can be considered a smooth continuous function in the continuum limit.   As discussed in reference \cite{phillips}, the continuum limit solution can be solved exactly for simple polygons by analyzing the three types of branches found in simple polygons (Case (1), (2) and (3) of Figure  \ref{fig:cases}(a),  \ref{fig:cases}(b), and  \ref{fig:cases}(c)).  For the Case (3) type, no disc centered on such a curve fills any more area than what is filled by placing two discs at the ends of the curve.  Thus in optimal solutions, Case (3) type curves are empty except for their ends.  

Let $\rho(t)$ represent the density of centers along a parameterized path of $M(G)$, $r(t)$ be the radius function, and $\kappa(t)$ be the local curvature of the path, where $(x(t), y(t))$ is the parameterization $t \in [t_a,t_b]$.  Given an expression for the unfilled area $A_i$ along the path $i$ of $M(G)$ of the form, 
\begin{eqnarray}
A_i = \int^{t_b}_{t_a} C_i(\kappa, r', r)\frac{dt}{\rho^2},
\label{eq:genarea}
\end{eqnarray}
where $C_i$ is a function to be determined, we would like to determine the function $\rho$ that minimizes this area constrained by
\begin{equation}
N =  \int^{t_b}_{t_a} \rho dt.
\label{eq:fixedN}
\end{equation}
Note that if we sum the unfilled areas $A_i$ over all of $M(G)$, then $\phi = 1-\sum(A_i/A_{G})$, where $A_{G}$ is the area of G.  
This variational problem can be solved by forming the Lagrangian
\begin{equation}
\mathcal{L}[\rho(t); \lambda] =   \int^{t_b}_{t_a} \left(C_i(\kappa, r', r)\frac{1}{\rho^2} + \lambda \rho \right) dt
\end{equation}
 and taking the pointwise derivative with respect to $\rho(t)$,
\begin{equation}
\frac{\partial \mathcal{L}}{\partial \rho} =  \int^{t_b}_{t_a} \left( \frac{-2C_i(\kappa, r', r)}{\rho^3} +\frac{\partial}{\rho^2\partial \rho} C_i(\kappa, r', r) +  \lambda  \right) \delta (t-\tau)dt.
\end{equation}
This relationship is satisfied by functions $\rho$ that satisfy 
\begin{equation}
-2C_i(\kappa, r', r) +\rho \frac{\partial}{\partial \rho} C_i(\kappa, r', r) +  \rho^3\lambda  = 0
\label{eq:lagrange}
\end{equation}
Solutions of the form $ \rho =  \left(\frac{C_i(\kappa, r', r)}{\lambda}\right)^{1/3}
\label{lagrangesolution}$
 satisfy this equation.  
 
 For Case (1), where $(x(t),y(t)) = At +B$, $r = r_0t$, and $t_a > 0$,
\begin{eqnarray}
C= {\left(1-r'^2\right)^{3/2}}/{\left(12r\right)}
\label{eq:nocurve}
\end{eqnarray}
as shown in Appendix A-1.  It follows that $\rho = \propto r^{-1/3}$.

For Case (2),  where $(x(t),y(t)) = (2r_0t,r_0t^2)$, $r = r_0(t^2 + 1)$, $r_0$ is the minimum of the radius function, and $\kappa(t) = \left({2r_0}\right)^{-1} \left(1+t^2\right)^{-3/2}$, 
\begin{equation}
C =  \frac{1}{12}\left(\frac{r_0\kappa}{r}\right)
=   \frac{1}{24r_0}\left(\frac{1}{1+t^2}\right)^{5/2}
\end{equation}
as shown in Appendix A-2.  It follows that  $\rho = \rho_0  \left(\frac{1}{1+t^2}\right)^{5/6}$ or  $\rho \propto r^{-5/6}$.  

For both Case (1) and Case (2), the distribution of centers follows a power law with respect to the local radius function.  Centers on $M(G)$ will be distributed more densely where the radius function is smaller.   
Given $\rho = \rho_0r^{-\alpha}$, for $\alpha$ = 1/3 or 5/6,  $\rho_0$ can be determined from Equation \ref{eq:fixedN},
\begin{eqnarray}
 \rho_0=&N \left(\int^{t_a}_{t_b}  r^{-\alpha} dt\right)^{-1} \\
  \rho_0=&N/R_0.
\end{eqnarray}
$R_0$ is then a constant determined by the radius function of the branch section of $M(G)$.
  
For Case (1), the distribution of centers on the medial axis path is also \emph{scale-free}.   The distribution of centers also follows a power law with respect to the distance from the vertex (where $t=0$) of the polygon.
  
Equation \ref{eq:genarea} becomes,
\begin{eqnarray}
A = \frac{1}{N^2}\int^{t_b}_{t_a} R_0^2 C(\kappa, r', r)dt =  \frac{1}{N^2} \mathcal{C}.
\end{eqnarray}

Thus, in the continuum limit the optimal filling solution converges to the area of the shape with an asymptotic error proportional to $N^{-2}$  for ideally distributed centers.  We presume that all shapes that can be approximated by simple polygons with an increasing number of sides also converge with an $N^{-2}$ error term. 

If we divide $M(G)$ into $k$ branch sections we can predict what fraction of the discs ($N_i/N$) will be distributed over each branch $i$ as $N\rightarrow\infty$.  
\begin{eqnarray}
A = \sum^k_1 A_i(N_i) \\ N = \sum^k_1 N_i
\end{eqnarray}
Since we have distributed our discs optimally, we can treat $A_i(N)$ as a continuous function and thus
\begin{equation}
\frac{\partial{A_i}}{\partial N_i} - \frac{\partial{A_j}}{\partial N_j} =0 , \forall j \neq i
\end{equation}

Arbitrarily setting $j=k$,

\begin{equation}
\frac{\partial{A_i}}{\partial N_i} - \frac{\partial{A_k}}{\partial N_k} =  -2\frac{\mathcal{C}_i}{N_i^3} +  2\frac{\mathcal{C}_k}{N_k^3} = 0
\end{equation}

\begin{equation}
N_i =  \left(\frac{\mathcal{C}_i}{\mathcal{C}_k}\right)^{1/3}N_k.
\end{equation}

The fraction of discs on a given branch $i$ is,
\begin{equation}
f_i = \frac{N_i}{N} = \frac{(\mathcal{C}_i)^{1/3}}{ (\mathcal{C}_1)^{1/3} + (\mathcal{C}_2)^{1/3} + ... + (\mathcal{C}_k)^{1/3}}. \label{gendist}
\end{equation}

For a triangle, which is always composed of three Case (1) branches, the fraction of the discs on a given path can be solved analytically to be
\begin{equation}
f_i =  \frac{\textrm{cot}(\theta_i)}{ \textrm{cot}(\theta_1)+ \textrm{cot}(\theta_2) +\textrm{cot}(\theta_3)} \label{triangledist}
\end{equation}
where $\theta_i$ is an internal angle of the triangle, each of which is associated with a branch.  From equation \ref{triangledist}, it is clear that the optimal solution preferentially populates medial axis branches associated with smaller internal angles.  This can be observed in the optimal filling of a triangle in Figure \ref{fig:triangle}.

\begin{figure*}[htb]
\begin{center}
\includegraphics[width=1.0\textwidth]{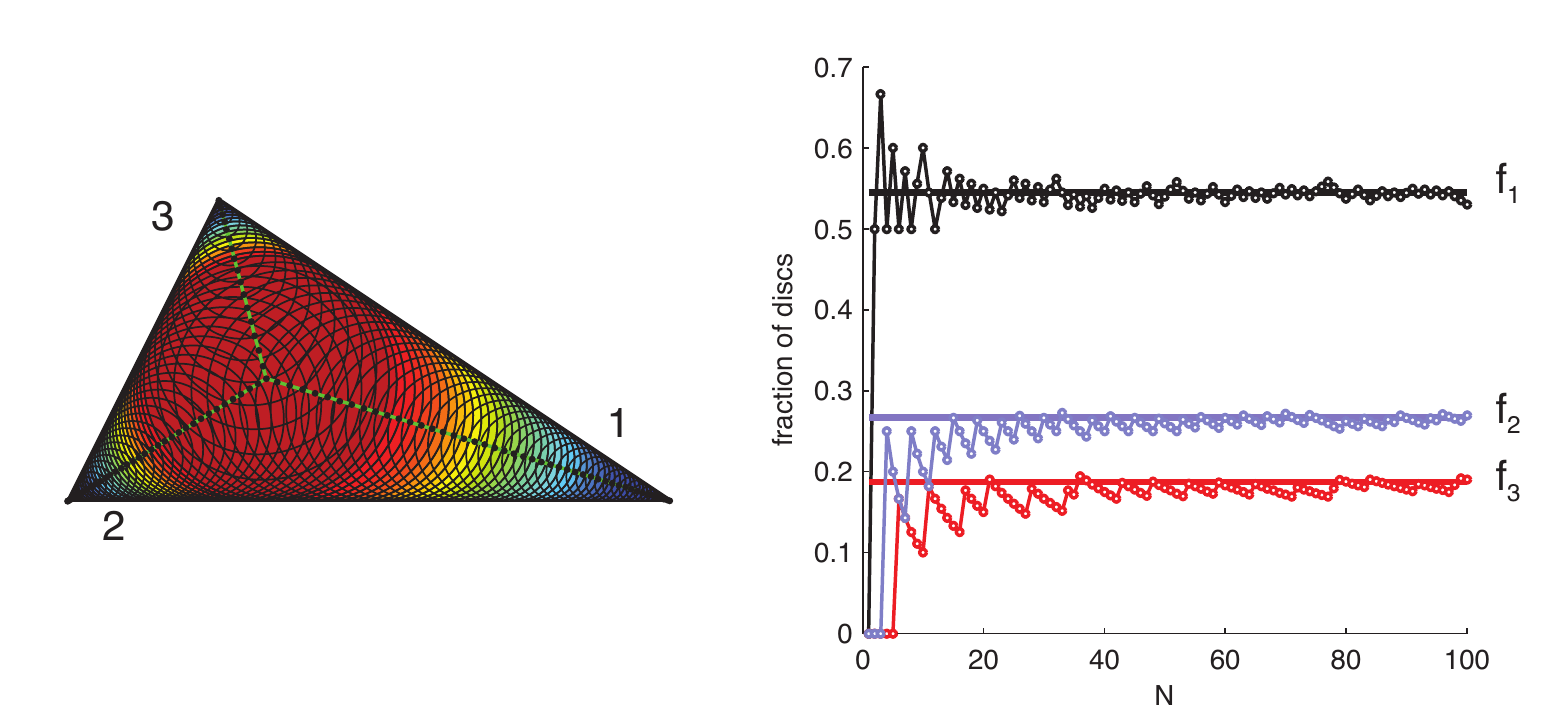}
\caption{The triangle on the left is filled with 100 discs.  On the right is the fraction of the discs  on each branch for $N$ = 1-100, compared to the prediction per equation \ref{triangledist}}.
\label{fig:triangle}
\end{center}\end{figure*}

An illuminating case to consider is a branch of $M(G)$ where both the curvature of the path and the radius function are constant. In Appendix A-3, we find
$C = \frac{1}{12}\left(\kappa^2 r + \frac{1}{r}\right)$, or a constant over the branch.   It immediately follows that $\rho = N/T$, where T is the length of the branch.  As expected, centers are distributed evenly over the branch.  For this case, $\mathcal{C} = T^3C$.    Thus, via equation \ref{gendist} $M(G)$ that can be divided into branches of constant curvature and radius functions also have known distributions as $N\rightarrow \infty$, and branches with higher curvatures will be more densely populated.

\section{Conclusion} \label{sec:conclusion}
In this paper we investigated the new problem introduced in reference reference \cite{phillips} of optimally filling shapes with balls of varying radius.   Filling combines two classic mathematical problems, the packing of shapes\cite{springerlink:10.1007/s00454-010-9254-3,springerlink:10.1007/s00454-010-9273-0,springerlink:10.1007/s00454-005-1198-7,springerlink:10.1007/s00454-005-1196-9,springerlink:10.1007/BF02574046,springerlink:10.1007/BF02187713} and the covering of space\cite{springerlink:10.1007/BF02187713,springerlink:10.1007/s00454-005-1202-2,springerlink:10.1007/s00454-009-9203-1,springerlink:10.1007/BF02187679,springerlink:10.1007/s00454-004-2916-2}.  Like in packing problems, the balls cannot overlap the boundary of the shape, but, like in covering, the balls may overlap each other without penalty.  This combination of constraints generates an interesting new problem.

In our research, filling solutions arise from the problem of modeling anisotropic nanoparticles as rigid bodies composed of a sum of isotropic volume-excluding potentials~\cite{PhysRevLett.95.056105,citeulike:1065652,doi:10.1021/nn901725b,doi:10.1021/nn203067y}.   Filling solutions have applications to many other areas of optimization, including the problem of irradiating a tumor with the fewest number of beam shots, while controlling the beam diameter, but without damaging surrounding tissue~\cite{Bourland}; using time-delayed sources to create shaped wavefronts; combining precision-placed explosives with tunable blast radii; positioning proximity sensors with defined radii; cell phone and wireless network coverage; or any problem of ablation or deposition where one has a sharp impenetrable boundary and a radially tunable tool.  

We find the filling problem to be surprisingly rich.  This paper describes the basic structure of the filling problem in arbitrary dimensions.  For polygons we have provided a deeper description of the solution space and detailed two methods for finding numeric approximations of the optimal filling solutions.  We also have shown how optimal solutions in polygons converge to simple analytical expressions as the number of discs approaches infinity.  

The solution space structure of a simple polygon has features that we expect to find in more generalized and higher dimensional shapes, namely first-order continuous manifolds that join at lower dimension manifolds where centers are trapped.  We predict that higher dimensional polytopes will also have manifolds with scale-free solutions.  

Even in two-dimensions, many open questions remain.  For example, how can optimal solutions be found for a generalized shape $G$ that is not a simple polygon?  Also, for a given shape $G$, is there a $N'$, such that for $N >N'$ junction points are always occupied by centers in optimal solutions?  Many of the potential applications of the filling problem demand solutions for three-dimensional shapes.  It is desirable to develop practical methods for finding optimal solutions in higher dimensions.  

\section{Acknowlegements}
We acknowledge Suresh Krishnan for software help and Greg Huber, Amir Haji-Akbari, and Michael Engel for interesting discussions. SCG and CLP acknowledge support by the U.S. Department of Energy, Office of Basic Energy Sciences,  Division of Materials Sciences and Engineering under Award DE-FG02-02ER46000. This research supported in part by the DOD/ASD(R\&E) under Award No. N00244-09-1-0062; any opinions, findings, and conclusions or recommendations expressed in this publication are those of the author(s) and do not necessarily reflect the views of the DOD/ASD(R\&E).  CLP also acknowledges support by the U.S. Department of Energy Computational Science Graduate Fellowship.

\pagebreak
\section{Glossary} \label{sec:glossary}

\paragraph{maximal ball}  A ball contained completely in a shape $G$ that is not a proper subset of any other ball also contained in $G$.  Also, a ball tangent to the surface of $G$ at at least two points, that is completely contained in $G$.  In a 2D planar shape, a maximal ball is a maximal disc.
\paragraph{medial axis} $M(G)$, the locus of the centers of all maximal balls of $G$.
\paragraph{radius function}  The radii of the maximal balls of a shape $G$.
\paragraph{normal point} A point on $M(G)$ that is the center of a maximal disc in contact with the boundary $S$ at exactly two separate but contiguous sets of points.
\paragraph{end point} A point on $M(G)$ that is the center of a maximal disc in contact with the boundary $S$ at exactly one contiguous sets of points.
\paragraph{branch point} A point on $M(G)$ that is the center of a maximal disc in contact with the boundary $S$ at three or more separate but contiguous sets of points.
\paragraph{branch}  A set of contiguous normal points on a medial axis.
\paragraph{parent of a branch} The two contiguous parts of $S$ from which the normal points of the branch are derived.  For a simple polygon, parents can be a polygon edge or a reflex point.
\paragraph{neighbor} If a maximal disc has a disc center that can be reached by a path along $M(G)$ starting at the center of maximal disc $A$ without traversing a third disc center, then it is the neighbor of maximal disc $A$.
\paragraph{center trap}  A point on $M(G)$ where a first-order discontinuity coupled with a local maximum in $\phi$ (all centers fixed) creates a local maximum that is stationary with respect to small changes in the position of the neighboring discs.
\paragraph{junction point} A point on $M(G)$ that can act as a center trap.  For a polygons, branch points are junctions.  Whether a generalized planar graph can have junction points that are not branch points we leave as an open question.
\paragraph{piece}  A junction point  or a section of a branch.
\paragraph{way}  A distribution of $N$ discs over the $K$ pieces that compose $M(G)$.
\paragraph{part} A connected set of pieces only connected to pieces not of the set by disc-center occupied junctions.

\pagebreak
\

\renewcommand{\figurename}{Fig. A.}
\setcounter{figure}{0}
\renewcommand{\theequation}{A.\arabic{equation}}
\setcounter{equation}{0}

\section{Appendix A-1: Distribution function along a medial axis branch with no curvature and a linear radius function}

\begin{figure}[h!]
\center
\includegraphics[width=0.75\textwidth]{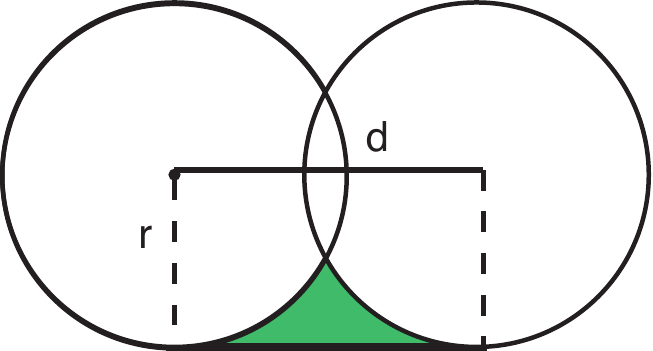}
\caption{The area shaded in green is the uncovered area between two discs of the same radius and the polygon edge. \label{fig:AreaNoCurvature}}
\end{figure}

We calculate the area between two discs along a medial axis branch generated by two polygon edge parents as the two discs approach each other.

Figure A.\ref{fig:AreaNoCurvature} shows two overlapping maximal discs, of the same radius, separated by a distance $d$, and one of the two lines tangent to both discs. The green region is the uncovered area in between the discs and the tangent line.   As $d\rightarrow 0$, what is the area, $A$ of the green region?  

\begin{eqnarray}
A &=& \textstyle{rectangle} - \textstyle{2 quarter circles} +{\textstyle\frac{1}{2}} \textstyle{lens}\\
   &=& dr - \frac{\pi}{2}r^2 + r^2 cos^{-1}\left(\frac{d}{2r}\right) - \frac{dr}{2}\sqrt{1-\left(\frac{d}{2r}\right)^2}
\end{eqnarray}

Using an acosine Taylor expansion and the square root
\begin{eqnarray}
A  &\simeq& dr - \frac{\pi}{2}r^2 + r^2 \left(\frac{\pi}{2} -\frac{d}{2r} -\frac{1}{6}\left(\frac{d}{2r}\right)^3  -\frac{3}{40}\left(\frac{d}{2r}\right)^5  \right) \\ 
&-& \frac{dr}{2}\left(1-\frac{1}{2}\left(\frac{d}{2r}\right)^2   -\frac{1}{8}\left(\frac{d}{2r}\right)^4\right) \\
A &=& \frac{d^3}{24r} + O(d^5)
\end{eqnarray}

We now approximate the uncovered area between two discs and the polygon edges for a radius function is that is not a constant, by bounding the answer between an upper and lower bound.
\begin{figure}[h!]
\center
\includegraphics[width=0.5\textwidth]{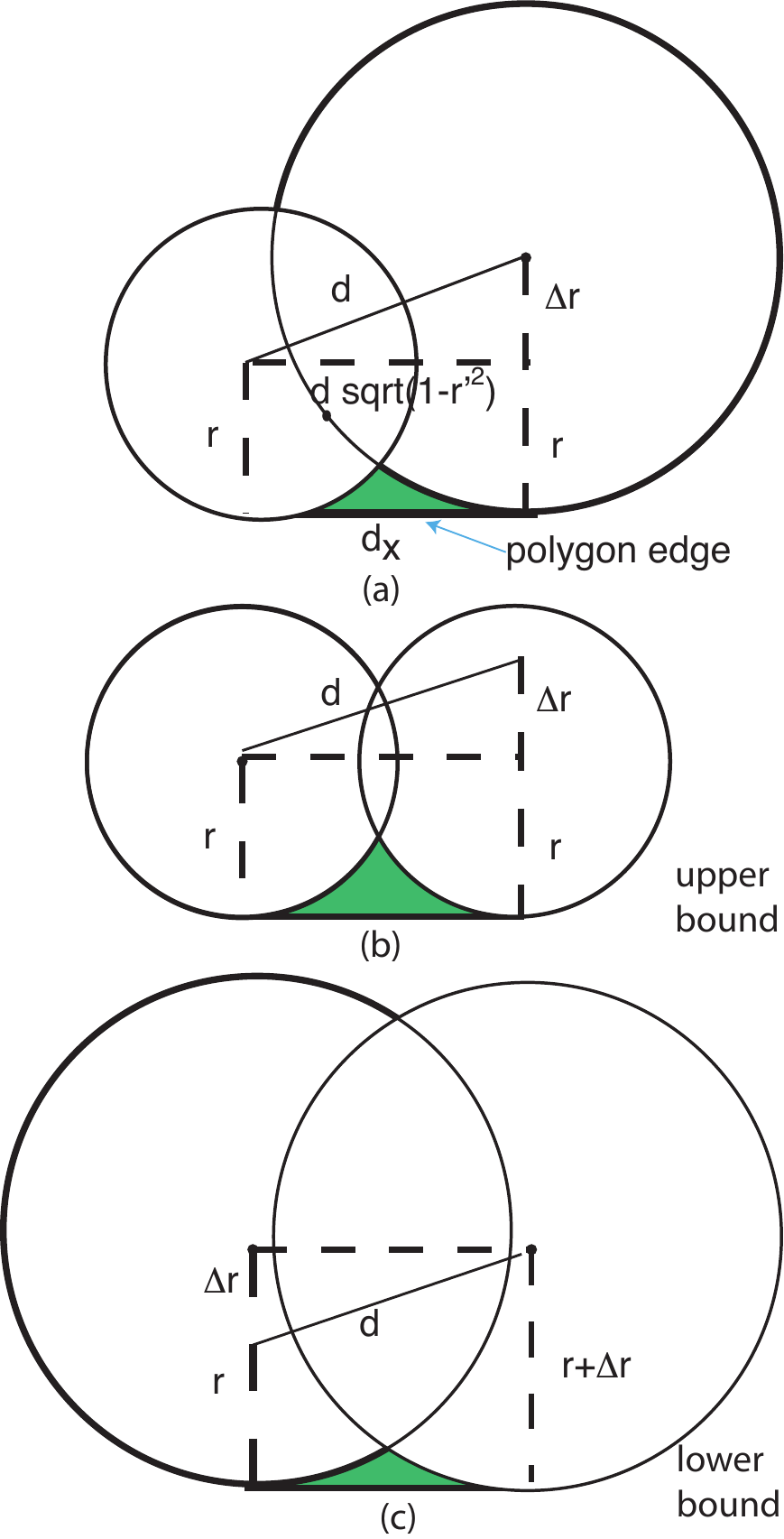}
\caption{(a) The area shaded in green is the uncovered area between two discs of different radius and the polygon edge.  (b) The area between two small circles and (c) two large circles provide and upper and lower limit for the shaded area. \label{fig:AreaNoCurvature_secondway}}
\end{figure}
Make one of the discs of Figure A.\ref{fig:AreaNoCurvature} larger by $\Delta r = dr'$, as per the Figure A.\ref{fig:AreaNoCurvature_secondway}(a).    The distance between the two centers is still defined as $d$.  
Put a cotangent disc of radius $R$ at the point of tangency of the both discs per Figure A.\ref{fig:AreaNoCurvature_secondway}(b) and Figure \ref{fig:AreaNoCurvature_secondway}(c).  If $R = r$ or if $R = r+dr'$, the centers are now $d\sqrt{1-r'^2}$ apart.  The uncovered area of  Figure A. \ref{fig:AreaNoCurvature_secondway}(a) is bound between  Figure A. \ref{fig:AreaNoCurvature_secondway}(b) and  Figure A. \ref{fig:AreaNoCurvature_secondway}(c) or between $\frac{1}{24}\frac{d^3(\sqrt{1-r'^2})^3}{r}$ and $\frac{1}{24}\frac{d^3(\sqrt{1-r'^2})^3}{r+dr'}$.  As $d\rightarrow 0$, the area uncovered is 
\begin{equation}
A_{uncovered} \simeq \frac{1}{24}\frac{d^3(\sqrt{1-r'^2})^3}{r}. \label{circlestangent}
\end{equation}
The area is then doubled to account for the identical uncovered piece on the other side due to the other tangent line (i.e. polygon edge).

We now observe that $d = \frac{1}{\rho}$ where $\rho$ is the density of disc centers along the branch.  To determine the total uncovered area along a branch of length $T$, we would sum all the uncovered areas that are at density $\rho$ along the branch, or
\begin{equation}
A = \int^T_0 \frac{\left(1-r'^2\right)^{3/2}\rho dt}{12r\rho^3} =  \int^T_0 \frac{\left(1-r'^2\right)^{3/2} dt}{12r\rho^2}
\end{equation}

\section{Appendix A-2: Distributions along the parabolic medial axis branch of a polygon }
\begin{figure}[h!]
\center
\includegraphics[width=0.5\textwidth]{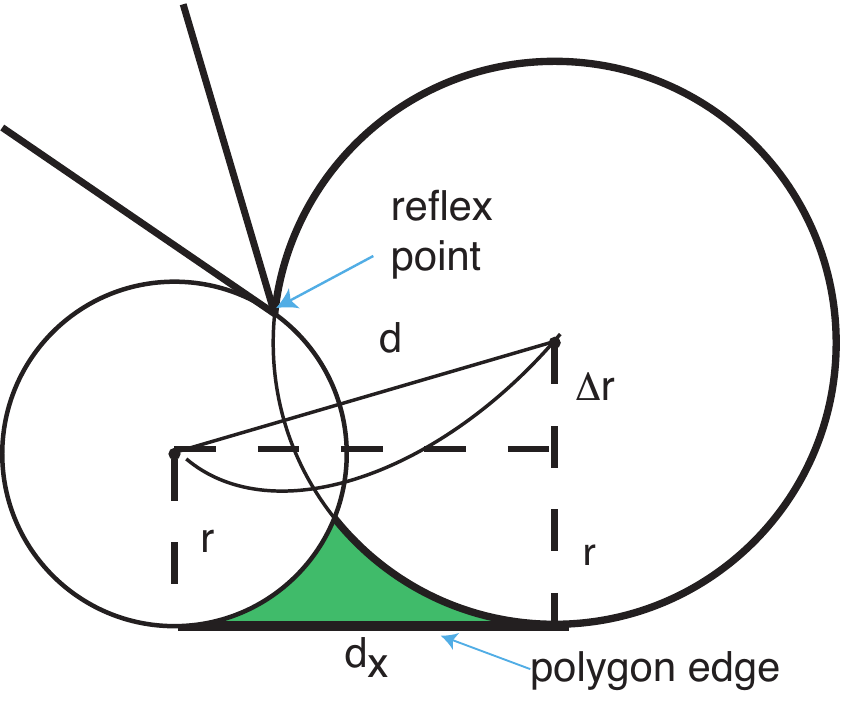}
\caption{The area shaded in green is the uncovered area between two discs of the different radius and the polygon edge along a parabolic path.\label{fig:Parabola}}
\end{figure}

For concave polygons, the medial axis branch associated with discs tangent to the reflex point and an edge of the polygon is a parabolic curve.  The reflex point forms the focus of the parabola and the polygon edge the directrix.  If the ends of such a parabolic curve are occupied, all the uncovered area is between the directrix and the set of overlapping discs distributed along the branch.  The parabola has both a changing radius function and changing curvature along the branch.  

First we note that for two discs that have centers of distance $d$ apart, the uncovered area between the discs and the directerix is the same as Equation \ref{circlestangent}.
\begin{equation}
 \textrm{Area}  \simeq \frac{d^3}{24r}\left(1-r'^2\right)^{3/2} 
\end{equation}
 
The local density of centers is  $\rho = 1/\Delta s$, where $\Delta s$ is the arc-length between the two centers.  However, as $d\rightarrow 0$, $d\approx \Delta s$ so,

\begin{equation}
 \textrm{Area}  \simeq \frac{1}{24\rho^3r}\left(1-r'^2\right)^{3/2} . \label{eqn:parabola}
\end{equation}

Using a parameterization of the parabola, where $y = at^2$, x = $2at$, $r = at^2 + a$, the arc length $s(t) = a \left(t\sqrt{1+t^2} + \textrm{sinh}^{-1}t \right)$,  and the curvature $\kappa(t) = \frac{1}{2a} \left(1+t^2\right)^{-3/2}$, then,

\begin{equation}
\frac{d r}{ds} = \frac{d r}{dt}\frac{dt}{ds} = \frac{2at}{2a\sqrt{1+t^2}} = \frac{t}{\sqrt{1+t^2}}
\end{equation}

Note that $r' < 1$, which is a general property of $r'$ of a medial axis.  

So,
\begin{equation}
(1-r'^2)^{3/2} = \left(\frac{1}{1+t^2}\right)^{3/2} = 2a\kappa.
\end{equation}

Now Equation \ref{eqn:parabola} is equal approximately to 
\begin{equation}
 \textrm{Area}  \simeq \frac{1}{24 \rho^3}\left(\frac{2r_0\kappa}{r}\right)
\end{equation}
where $r_0$ is the smallest radius of the parabola, or $r_0 = a$.  Or,

\begin{equation}
 \textrm{Area}  \simeq \frac{1}{24r_0 \rho^3}\left(\frac{1}{1+t^2}\right)^{5/2}.
\end{equation}

If the parabola is defined from $t_a$ to $t_b$, then the total uncovered area is 
\begin{equation}
\int^{t_b}_{t_a} \textrm{Area} \cdot \rho dt =  \int^{t_b}_{t_a}  \frac{1}{24r_0 \rho^2}\left(\frac{1}{1+t^2}\right)^{5/2}dt.
\end{equation}

\section{Appendix A-3: Constant curvature, constant radius function}
Consider the case of two identical discs separated by a branch of length $d$ and curvature $\approx \kappa$.  What is the uncovered area between them?  See  Figure A.\ref{construct}.  The uncovered area is the area swept between the two arcs tangent to the discs minus the area covered by the discs within the area swept, or twice the green area of  Figure A.\ref{construct}b.    Let $\theta$ be the angle $\angle ABC$, $\theta_c$ be angle $\angle DCE$.  Let $d_{cl}$ be the chord length between $A$ and $C$.  Then $\theta = d\kappa$, $d_{cl}$ = $\frac{2}{\kappa}\textrm{sin}\left(\frac{d\kappa}{2}\right)$, and $\theta_c = 2 \textrm{cos}^{-1}\left( \frac{1}{\kappa r} \textrm{sin}\left(\frac{d\kappa}{2}\right)\right)$.    The area of the green region of  Figure A.\ref{construct}b is equal to the half disc minus the grey area, or $\frac{\pi r^2}{2} - \frac{r^2}{2}(\theta_c - sin(\theta_c))$.
 
  \begin{figure}[h!]
\begin{center}
\includegraphics[width=3in]{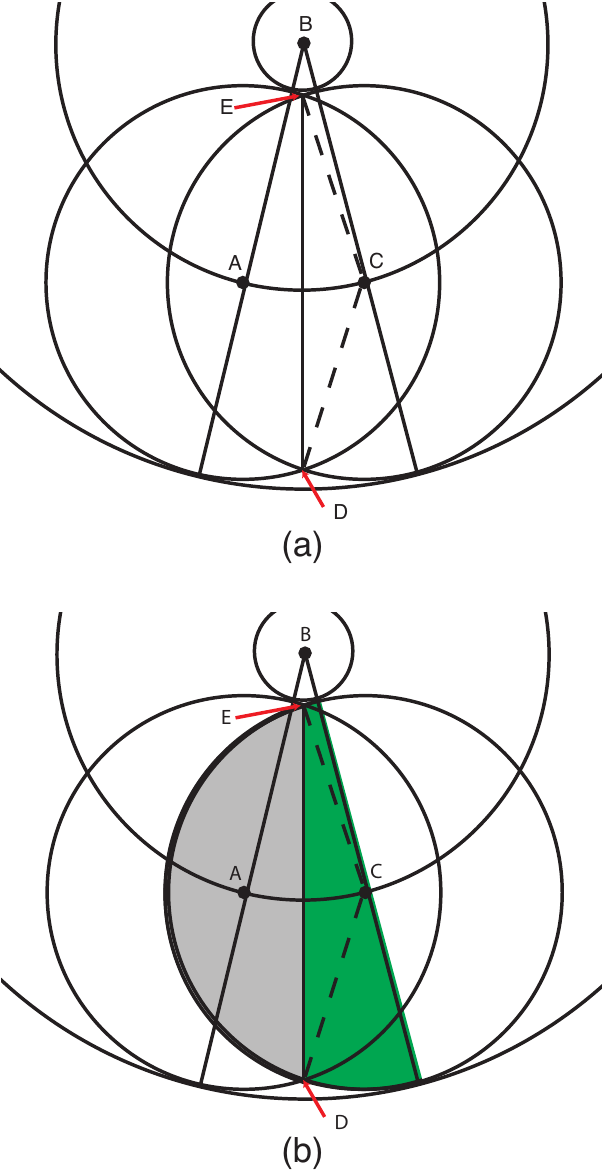}
\caption{The uncovered area between two discs on a branch of constant radius function, constant curvature.}
\label{construct}
\end{center}\end{figure}

 \begin{eqnarray}
 A_{uncovered} &=& \left(\left(\frac{1}{\kappa }+ r\right)^2 - \left(\frac{1}{\kappa }- r\right)^2\right)\frac{d\kappa}{2} - \pi r^2 + r^2(\theta_c - \textrm{sin}(\theta_c)) \\
 & =& 2rd - \pi r^2 + r^2(\theta_c - \textrm{sin}(\theta_c))
 \end{eqnarray}
 
 Expanding the third term
 
\begin{eqnarray}
&r^2\theta_c - r^2\textrm{sin}(\theta_c) \\
&2r^2 \textrm{cos}^{-1}\left( \frac{1}{\kappa r} \textrm{sin}\left(\frac{d\kappa}{2}\right)\right) - r^2 \textrm{sin}\left(2 \textrm{cos}^{-1}\left( \frac{1}{\kappa r} \textrm{sin}\left(\frac{d\kappa}{2}\right)\right)\right) \\
&2r^2 \textrm{cos}^{-1}\left( \frac{1}{\kappa r} \textrm{sin}\left(\frac{d\kappa}{2}\right)\right) - 2r^2 \textrm{sin}\left( \textrm{cos}^{-1}\left( \frac{1}{\kappa r} \textrm{sin}\left(\frac{d\kappa}{2}\right)\right)\right) \left( \frac{1}{\kappa r} \textrm{sin}\left(\frac{d\kappa}{2}\right)\right)\\
&2r^2\left( \textrm{cos}^{-1}\left( \frac{1}{\kappa r} \textrm{sin}\left(\frac{d\kappa}{2}\right)\right) - \sqrt{1 -\left( \frac{1}{\kappa r} \textrm{sin}\left(\frac{d\kappa}{2}\right)\right)^2} 
\left( \frac{1}{\kappa r} \textrm{sin}\left(\frac{d\kappa}{2}\right)\right)
\right)
\end{eqnarray}

Substituting a Taylor series expansion for the inverse cosine and square root
\begin{eqnarray}
&2r^2\left( \pi/2 - 
\left( \frac{1}{\kappa r} \textrm{sin}\left(\frac{d\kappa}{2}\right)\right) -
\frac{1}{6}\left( \frac{1}{\kappa r} \textrm{sin}\left(\frac{d\kappa}{2}\right)\right)^3 \right)-\\
& 2r^2\left(1 -\frac{1}{2}\left( \frac{1}{\kappa r} \textrm{sin}\left(\frac{d\kappa}{2}\right)\right)^2 \right)
\left( \frac{1}{\kappa r} \textrm{sin}\left(\frac{d\kappa}{2}\right)\right)
 \\
&=2r^2\left( \pi/2 - 
\left( \frac{2}{\kappa r} \textrm{sin}\left(\frac{d\kappa}{2}\right)\right) +
\frac{2}{3}\left( \frac{1}{\kappa r} \textrm{sin}\left(\frac{d\kappa}{2}\right)\right)^3 \right)
\end{eqnarray}
Thus,
\begin{eqnarray}
A_{uncovered} &=&  2rd  - 
 \frac{4 r}{\kappa} \textrm{sin}\left(\frac{d\kappa}{2}\right) +
\frac{1}{\kappa^3 r} \frac{4}{3}\left( \textrm{sin}\left(\frac{d\kappa}{2}\right)\right)^3 
\end{eqnarray}

Substituting a Taylor series expansion for the sine terms.
\begin{eqnarray}
A_{uncovered} &=&  2rd  - 
 \frac{4 r}{\kappa} \left( \left(\frac{d\kappa}{2}\right) - \frac{1}{6}\left(\frac{d\kappa}{2}\right)^3\right) +
\frac{1}{\kappa^3 r} \frac{4}{3}\left(\frac{d\kappa}{2}\right)^3 \\
&=& \frac{1}{12} d^3\kappa^2 r + \frac{1}{12}\frac{d^3}{r}
\end{eqnarray}

So the uncovered area along the whole length of the branch is  ($d = \frac{1}{\rho}$)

\begin{equation}
A = \int^T_0 \frac{1}{12}\left(\kappa^2 r + \frac{1}{r}\right)\frac{dt}{\rho^2} = \frac{T^3}{12N^2}\left(\kappa^2 r + \frac{1}{r}\right)
\end{equation}
where $N$ is the total number of discs distributed over the branch section.
Using equations (10) and (14) of the main paper, if $M(G)$ can be broken into branch sections with approximately constant $r$ and curvature, then the fractional distribution of the $N$ discs over each section can be determined such that the fraction $f_k$ of discs distributed over a section $k$ of length $T_k$ is 

\begin{equation}
f_k  \propto  T_k\left(  \kappa_k^2r_k +\frac{1}{r_k}    \right)^{1/3}  
\end{equation}

We observe that, in general, the density of discs is higher in regions of high curvature.  

\section{Appendix A-4: How many ways?} 

Per conjecture \ref{conj:onemax}, to find the optimal $N$ filling solution, a maximum must be generated for every way of $N$ discs and $K$ pieces.   How many maxima searches is this?   Assume that the $K$ pieces have $J$ junctions, $J < K$.    For $m \in \mathbb{N}, 0 \leq m \leq J$, there are ${J \choose m}$ ways to occupy the junctions, leaving $N-m$ remaining discs to allocate over the $K-J$ remaining pieces.    A weak composition is a way of partitioning an integer into a sequence of non-negative integers, where order matters.  The number of weak compositions of $N-m$ discs over $K-J$ pieces is ${{N-m+K-J-1} \choose {K-J-1}}$.

Thus the number of ways to be searched is equal to  $\sum_{m=0}^{min(J,N)}{J \choose m} {{N-m+K-J-1} \choose {K-J-1}}$.

\begin{eqnarray}
&&\sum_{m=0}^{min(J,N)}{J \choose m} {{N-m+K-J-1} \choose {K-J-1}}\\ 
&=& \sum_{m=0}^{J} \frac{J! (N-m+K-J-1)!}{m! (J-m)! (K-J-1)! (N-m)!} \\
&=& O\left(N^{K-J-1}\right)
\end{eqnarray}


\bibliographystyle{spphys}       


\end{document}